\documentclass[10pt]{amsart}

\usepackage[margin=1in, marginparsep=1pt,
            marginparwidth=124pt]{geometry}

\usepackage[english]{babel}
\usepackage[usenames,dvipsnames,svgnames,table]{xcolor}
\usepackage[pagebackref,linktocpage=true,colorlinks=true,linkcolor=Blue,citecolor=BrickRed,urlcolor=RoyalBlue]{hyperref}
\usepackage[alphabetic]{amsrefs}

\usepackage[colorinlistoftodos,prependcaption,textsize=tiny]{todonotes}
\usepackage{amsmath,amssymb,amsfonts,amsthm,enumerate,textcomp}
\usepackage{url}
\usepackage{graphicx}
\usepackage{xcolor}
\usepackage{wrapfig}

\usepackage[mathscr]{euscript} % provides the M for the manifold

\usepackage{esint} 
\usepackage{enumitem}

\theoremstyle{plain}
\newtheorem{theorem}{Theorem}[section]
\newtheorem{definition}[theorem]{Definition}
\newtheorem{lemma}[theorem]{Lemma}
\newtheorem{prop}[theorem]{Proposition}

\usepackage{amsmath}
\usepackage{amssymb}

\numberwithin{equation}{section}

\newcommand\R{\mathbb R} % real line
\newcommand\N{\mathbb N} 
\newcommand\Z{\mathbb Z}
\renewcommand\div{\mbox{div}}

\newcommand{\dist}{{ \rm dist}}
\newcommand\p{\partial}
\newcommand\wh{\widehat}

\newcommand\e{\varepsilon}
\renewcommand{\epsilon}{\e}

\renewcommand\H{\mathscr H}

\newcommand{\I}[1]{\chi_{\{#1>0\}}}
\newcommand{\po}[1]{\{#1>0\}}
\newcommand{\fb}[1]{\partial\{#1>0\}}

\newcommand\Om{\Omega}
\newcommand\na{\nabla}
\usepackage{tikz}
\usetikzlibrary{arrows,shapes,positioning}
\usetikzlibrary{decorations.markings}
\tikzstyle arrowstyle=[scale=1]
\tikzstyle directed=[postaction={decorate,decoration={markings,
    mark=at position .65 with {\arrow[arrowstyle]{stealth}}}}]
\tikzstyle reverse directed=[postaction={decorate,decoration={markings,
    mark=at position .65 with {\arrowreversed[arrowstyle]{stealth};}}}]

\allowdisplaybreaks


\begin{document}

\title[Limit behaviour of a singular perturbation problem]{Limit behaviour of a singular perturbation problem \\
for  the biharmonic operator}\thanks{The first and third authors are member of INdAM. This work has been supported by the Australian
Research Council Discovery Project DP170104880
NEW ``Nonlocal Equations at Work'',
the Australian Research Council DECRA DE180100957
``PDEs, free boundaries and applications'' and the Fulbright Foundation.}

\author{Serena Dipierro}
\address{Serena Dipierro: Department of Mathematics
and Statistics,
University of Western Australia,
35 Stirling Hwy, Crawley WA 6009, Australia}
\email{serena.dipierro@uwa.edu.au}

\author{Aram L. Karakhanyan}
\address{Aram Karakhanyan:
School of Mathematics, The University of Edinburgh,
Peter Tait Guthrie Road, EH9 3FD Edinburgh, UK}
\email{aram.karakhanyan@ed.ac.uk}

\author{Enrico Valdinoci}
\address{Enrico Valdinoci:
Department of Mathematics
and Statistics,
University of Western Australia,
35 Stirling Hwy, Crawley WA 6009, Australia}
\email{enrico.valdinoci@uwa.edu.au}

\begin{abstract} We study here a singular perturbation
problem of biLaplacian type, which
can be seen as the
biharmonic counterpart of classical combustion models.

We provide different results, that include
the convergence to a free boundary problem driven by a biharmonic
operator, as introduced in~\cite{BIHA}, and
a monotonicity formula in the plane.
For the latter result, an important tool is provided by an integral identity
that is satisfied by solutions of the singular perturbation
problem.

We also investigate
the quadratic behaviour of
solutions near the zero level set, at least
for small values of the perturbation parameter.

Some counterexamples to the uniform regularity are also provided
if one does not impose some structural assumptions on the forcing term.
\end{abstract}

\date{\hbox{\today}}
\maketitle

\setcounter{tocdepth}{2}
\tableofcontents

\section{Introduction}
\label{sec-intro}  

In this article we study bounded solutions $\{u^\e\}_{\e>0}$
of the singularly perturbed biLaplacian equation
\begin{equation}\label{eq-sing-pert}
2\Delta^2 u^\e=-\beta_\e\left( {u^\e}\right) \quad {\mbox{ in }}\Omega,
\end{equation}
where~$\e\in(0,1]$ is a small parameter, $\Omega$ is a smooth and bounded
domain of~$\R^n$,
\begin{equation}\label{betaeps}
\beta_\epsilon(t):=
\frac1{\epsilon}\,\beta\left(\frac{t}{\epsilon}\right),\end{equation}
and~$\beta$ is a smooth, nonnegative function,
with support contained in~$[0, 1]$ 
and such that 
\begin{equation}\label{lambda}\int_{\R}\beta(t)\,dt=
\int_{0}^1\beta(t)\,dt =1.\end{equation}

Equation~\eqref{eq-sing-pert} can be seen as the
biharmonic counterpart of classical combustion models,
see e.g.~\cite{MR1900562}.
We observe that the problem in~\eqref{eq-sing-pert}
is variational, and indeed
solutions of~\eqref{eq-sing-pert}
are critical points of the functional
\begin{equation}\label{DOBB}
J_\e[v]:=\int_\Om |\Delta v(x)|^2 +\mathcal B_\e\big(v(x)\big)\,dx,
\end{equation}
where
\begin{equation}\label{bepsilon}
\mathcal B_\e(v):=\int_0^v\beta_\e(t)\,dt.
\end{equation}
The factor~$2$ in equation~\eqref{eq-sing-pert} has been
placed exactly to avoid additional factors~$1/2$ in the energy functional~\eqref{DOBB}
(and thus to make the comparison with the existing literature
more transparent).
As a special example, one can consider minimizers
of~$J_\e$
with respect to Navier boundary conditions,
that is, given~$u_0\in W^{2,2}(\Omega)$, one can minimize~$J_\e$
among the
set of competitors
given by
$$ {\mathcal{A}}:=\Big\{ u\in W^{2,2}(\Omega) {\mbox{ s.t. }} u-u_0\in W^{1,2}_0(\Omega)
\Big\}.$$
Then, minimizers of~\eqref{DOBB}
are taken in the class~${\mathcal{A}}$ and they
are solutions
of~\eqref{eq-sing-pert} with boundary data~$u=u_0$ and~$\Delta u=0$
along~$\partial\Omega$. 
See for instance the ``hinged problem'' on the right hand side of Figure 1(a)
and on page 84 of~\cite{SW}, or Figure 1.5 on page 6 of~\cite{GANGULI},
or the monograph~\cite{GAZ} for further information of this type
of boundary conditions.

The existence of minimizers of the functional in~\eqref{DOBB}
in the class~${\mathcal{A}}$
is obtained by the direct methods in the calculus of variations,
see Lemma~2.1 in~\cite{BIHA}.\medskip

Some motivations for investigating equations involving the
biharmonic operator come from
classical models 
for rigidity problems, which have concrete applications,
for example, in the construction of suspension bridges,
see e.g.~\cite{MR866720}
and the references therein. See also
formula (1) in~\cite{MR3512704} and the references therein for other
classical applications of the biharmonic
operator in the study of steady state incompressible fluid flows at
small Reynolds numbers under the Stokes flow approximation assumption.
In our framework, we will present a simple game-theoretical model
for the problem in~\eqref{eq-sing-pert} in Section~\ref{sec:mot}.
\medskip

The minimizers of~$J_\e$ enjoy suitable regularity and compactness properties,
and they are related to a free boundary problem of biharmonic type
which has been recently investigated in~\cite{BIHA}.
To formalize this, we consider the functional
\begin{equation}\label{J}
J[v]:= \int_\Om |\Delta v(x)|^2 +\chi_{(0,+\infty)}\big(v(x)\big)\,dx.\end{equation}

Though free boundary problems are by now a classical topic of investigation
(see~\cite{MR618549}), 
the setting of higher order operators provides only few results available, and
the analysis of the free boundary problem in~\eqref{J} has been only recently
initiated in~\cite{BIHA} (see also~\cite{mawi} where
other types of free boundary problems for higher order operators
have been considered). 
Furthermore, obstacle problems
involving biharmonic operators have been studied
in~\cite{caffa, MR620427, MR705233, pozzo, novaga1, novaga2}.

In this setting, one can relate minimizers of the functional~$J_\e$
in~\eqref{DOBB} with minimizers of the free boundary problem in~\eqref{J},
according to the following convergence result:

\begin{theorem}\label{thm-eps-compactness}
Let $\{u^\e\}$ be a family of minimizers of the functional~$J_\e$, as defined
in~\eqref{DOBB}, with
$$ \sup_{\e\in(0,1]} \|u^\e\|_{L^\infty(\Omega)}<+\infty
. $$
Then, as~$\e\to0^+$, up to a subsequence,
\begin{itemize}
\item $u^{\e}\to u$ locally uniformly in $C^{1, \alpha}_{{\rm loc}}(\Om)$ for any~$\alpha\in(0,1)$,
\item $u^{\e}\to u$ in $W^{2, p}_{{\rm loc}}(\Om),$ for every $p>2$,
\item $\Delta u^{\e}\to \Delta u$ in $BMO$,
\item $u$ is a minimizer of the functional~$J$, as defined in~\eqref{J}.
\end{itemize}
\end{theorem}

\medskip

We observe that solutions of~\eqref{eq-sing-pert},
and in particular minimizers of~$J_\e$, naturally develop a notion of limit
free boundary. Indeed, if~$u^\e$ 
is a minimizer of~$J_\e$ which approaches~$u$ as~$\e\to0^+$,
one is interested in the geometric properties of the set~$\partial\{ u^\e>0\}$.
To analyze and classify this type of sets, it would be extremely desirable to have suitable
monotonicity formulas. 
Differently from the classical case in which the equation is of second order
(i.e., the energy functional is induced by the classical Dirichlet form,
see~\cite{MR618549}),
in our setting no general result of this type is available in the literature.

In our framework, we will obtain a monotonicity formula, relying 
on the following integral equation for
solutions of~\eqref{eq-sing-pert}:

\begin{lemma} \label{IDENDOM}
Let~$u^\e$ be
a solution of~\eqref{eq-sing-pert}. Then,
for any~$\phi=(\phi^1,\dots,\phi^n)\in C^\infty_0(\Omega,\R^n)$,
\begin{equation}\label{DOMAIN} 2\int_\Omega  \Big(2\,{\rm tr} \big(D^2u^\e(y)\,D\phi(y)\big)
+\nabla u^\e(y)\cdot\Delta\phi(y)
\Big)\,\Delta u^\e(y)\,dy
=\int_\Omega {\rm div}\phi(y)\,\Big( |\Delta u^\e(y)|^2+ \mathcal B_\e\big(u^\e(y)\big)\Big)
\,dy.\end{equation}
\end{lemma}

With this, 
the argument leading to the monotonicity formula is based on
the choice
of a test function~$\phi$ in~\eqref{DOMAIN} with a particular form,
see~\cite{MR1620644}.
More precisely, we focus on the two-dimensional case
and we prove the following 

\begin{theorem}\label{lemma:F}
Let~$n=2$ and~$\tau>0$ such that~$B_\tau\Subset\Omega$.
Let~$u^\e$
be a solution of~\eqref{eq-sing-pert}, with
\begin{equation}\label{45:190}
u^\e(0)=0\qquad{\mbox{and}}\qquad\nabla u^\e(0)=0.
\end{equation}
Then, there exists a function~$E^\e:(0,\tau)\to\R$, which is bounded
in~$(0,\tau)$,
nondecreasing and
such that, for any~$\tau_2>\tau_1>0$,
\begin{equation}\label{MONOFORMULA}
E^\e(\tau_2)-E^\e(\tau_1)=
\int_{\tau_1}^{\tau_2}\left\{
\frac1{r^2}
\int_{\p B_r}\left[\left( \frac{u^\e_{\theta r}}r-\frac{2 u^\e_r}{r^2} \right)^2+
\left(u^\e_{rr}-\frac{3u^\e_r}r+4\frac{ u^\e}{r^2}\right)^2\right]\right\}.\end{equation}
\end{theorem}

Theorem~\ref{lemma:F} can be also made more precise, since
the function~$E^\e$ is given explicitly by
\begin{equation}\label{EMMEDE}\begin{split}&
E^\e(r) = 
\int_{\partial B_r}
\left(\frac{\Delta u^\e\,u^\e_r}{2r^2}-
\frac{5(u^\e_r)^2}{2r^3}
-\frac{\Delta u^\e u^\e}{r^3}+\frac{6u^\e u^\e_r}{r^4}+
\frac{u^\e_{\theta}u^\e_{\theta r}}{r^4}
-\frac{4(u^\e)^2}{r^5}
-\frac{3(u_\theta^\e)^2}{2r^5}
\right)\\&\qquad\qquad\qquad
+
\frac1{4r^2}\int_{B_r} \big( |\Delta u^\e|^2+\mathcal{B}_\e(u^\e)\big)+
\int_{0}^{r}\frac1{\rho^{3}}\int_{B_\rho} \beta_\e(u^\e)\,u^\e
.
\end{split}
\end{equation}

The proof of Theorem~\ref{lemma:F}
is based on a series of careful integration by parts
aimed at spotting suitable integral cancellations, 
which are possible in dimension~$2$. In addition,
some ``high order of differentiability'' terms naturally appear in the computations,
which need to be suitably removed in order to rigorously make sense
of the formal manipulations.
\medskip

In light of Theorem~\ref{lemma:F}, one can pass
to the limit and obtain a monotonicity formula for weak solutions
of the limit free boundary problem in~\eqref{J}. This result
extends the monotonicity formula found in~\cite{BIHA}
for the case of minimizers to the more general setting of weak solutions.
To this end, we introduce the following setting.

\begin{definition} A function
$u\in W^{2, 2}(\Om)$ is said to be a weak solution of the free boundary problem in~\eqref{J}
if 
\begin{eqnarray*}
&& {\mbox{$\Delta^2 u=0$ in $\po u\cup\{u<0\}$,}} \\
&&{\mbox{\eqref{DOMAIN} holds,}}\\
{\mbox{and }}&&{\mbox{$\fb u$ is locally rectifiable, i.e.
$\fb u=M_0\cup\left(\displaystyle\bigcup_{k=1}^\infty M_k\right)$,}}\\&&\qquad{\mbox{ where 
$M_k, k\ge 1$ are $C^1$ hypersurfaces and $\H^{n-1}(M_0)=0$.  }}\end{eqnarray*}
\end{definition}

To formulate next result, we also let 
\begin{equation}\label{la u}\begin{split}&
{\mbox{$u^{\e_j}$ be a sequence of solutions of~\eqref{eq-sing-pert}
satisfying~\eqref{45:190}, with~$\e_j\searrow0$ as~$j\to+\infty$,}}\\&
{\mbox{and $u$ be a limit of $u^{\e_j}$ 
such that $u^{\e_j}\to u\ge 0$ uniformly,}}\end{split}\end{equation}
and we define
\begin{equation}\label{67A02333}
\begin{split}
& E(r) := 
\int_{\partial B_r}
\left(\frac{\Delta u\, u_r}{2r^2}-
\frac{5 u_r^2}{2r^3}
-\frac{\Delta u \, u}{r^3}+\frac{6 u \, u_r}{r^4}+
\frac{ u_{\theta}\, u_{\theta r}}{r^4}
-\frac{4 u^2}{r^5}
-\frac{3 u_\theta^2}{2r^5}
\right)
+
\frac1{4r^2}\int_{B_r} \big( |\Delta  u|^2+\mathcal{\wh B}\big),
\end{split}
\end{equation}
where $ \mathcal{\wh  B}$ is the weak star limit of $\mathcal B_\e(u^\e)$. 
In this setting, we have the following monotonicity formula:

\begin{theorem}\label{thm:strong-lap}
Let $u$ be a weak solution satisfying~\eqref{la u}.
Suppose that
\begin{equation}\label{Cg09}
\lim_{\e_j\to 0} |\{0<u^{\e_j}\le\e_j\}\cap B|=0\end{equation} for every ball $B\Subset \Om$
and 
\begin{equation}\label{QUAD-lap}
\|\Delta u^{\e_j}\|_{L^\infty(\overline{B_\tau})}
\le C
,
\end{equation}
for some~$C>0$ independent of~$j$.

Then,
\begin{equation}\label{STRONG:C}
{\mbox{ $\Delta u^{\e_j}\to \Delta u $ strongly in $L^2_{loc}(B_\tau)$. }}\end{equation}
In particular, 
\begin{equation}\label{STRONG:H}
{\mbox{the
energy identity in~\eqref{DOMAIN} holds for $ u$,  
}}\end{equation}
with $\mathcal B_\e$ replaced by $\mathcal{\wh B}$, the weak star limit of $\mathcal B_{\e_j}$.  

Furthermore,
if 
\begin{equation}\label{QUAD}
|D^2 u(x)|
\le C\, \qquad{\mbox{for any }}x\in B_1,
\end{equation}
then
\begin{equation}\label{STRONG:EE}
{\mbox{for almost every $t\in(0, \tau)$, the function~$ E$ is well defined and non-decreasing.}}
\end{equation}

Moreover, if~$\tau_2>\tau_1>0$ and~$B_{\tau_2}\Subset\Omega$, then
\begin{equation}\label{VSmciiririMONOFORMULA}
E(\tau_2)-E(\tau_1)=
\int_{\tau_1}^{\tau_2}\left\{
\frac1{r^2}
\int_{\p B_r}\left[\left( \frac{u_{\theta r}}r-\frac{2 u_r}{r^2} \right)^2+
\left(u_{rr}-\frac{3u_r}r+4\frac{ u}{r^2}\right)^2\right]\right\}.\end{equation}

In addition,
\begin{equation}\label{VSmciiririMONOFORMULA:2}\begin{split}&
{\mbox{if~$ E$ is constant in~$(0,\tau)$,
then the function~$\displaystyle
-\frac{u_r}{r}+2 \frac{ u}{r^2}$ is constant in~$B_\tau$,}}\\
&{\mbox{and moreover~$u$
is a homogeneous function of degree two in~$B_\tau$. }}\end{split}\end{equation}

Finally, for every sequence $r_k\searrow 0$ there exists a subsequence 
$r_{k_j}$ such that 
\begin{equation}\label{thm:strong-lap-EQ}
{\mbox{the scaled functions $\frac{u(r_{k_j}x)}{r_{k_j}^2}$ either converge to zero or to
a homogeneous function of degree two. }}\end{equation}
\end{theorem}

We point out that condition~\eqref{QUAD} ensures that~$ E$ remains bounded
as~$r\to0^+$. 
\medskip

It is interesting to detect the quadratic behaviour of
solutions of~\eqref{eq-sing-pert} near the zero level set, at least
for small values of~$\e$. To this end, we provide this limit result:

\begin{theorem}\label{ALFAGA}
Let $u^{\e}$ be a sequence of solutions to \eqref{eq-sing-pert}
in~$\Omega$. Let
$0\in \Omega$, and~$\alpha$, $\gamma\in\R$. Suppose that
\begin{equation}\label{CONW22}
{\mbox{$u^{\e}$
converge to~$u:=\displaystyle\frac\alpha2\, (x_1)_+^2+\displaystyle\frac\gamma2\,(x_1)_-^2$,
as~$\e\to0^+$, up to a subsequence,
in~$W^{2,2}_{\rm loc}(\Omega)$. }}\end{equation} Then:
\begin{itemize}
\item If~$\alpha$, $\gamma>0$, we have that 
\begin{equation}\label{ARRCLA:1}
\alpha=\gamma.
\end{equation}
\item If~$\alpha$, $\gamma<0$, we have that 
\begin{equation}\label{ARRCLA:1bis}
\alpha=\gamma.
\end{equation}
\item If~$\alpha>0$, $\gamma\le0$, we have that
\begin{equation}\label{ARRCLA:2}\alpha^2-\gamma^2=1.
\end{equation}\end{itemize}
Moreover,
\begin{equation}\label{ARRCLA:4}
{\mbox{the case~$\alpha<0$, $\gamma=0$
cannot hold.}}
\end{equation}
\end{theorem}

\medskip

We also observe that, in general, one cannot
expect uniform second derivative bounds
on solutions of~\eqref{eq-sing-pert}
without any additional structure (not even in low
dimension). For this,
we provide the following one-dimensional counterexample, where
the forcing term~$\beta_\epsilon$ satisfies~\eqref{lambda},
but does not fulfill the structural assumption in~\eqref{betaeps}:

\begin{theorem}\label{CONTR}
There exists~$\delta>0$ such that for all~$\e>0$ sufficiently small there
exist~$
\beta^\e\in C^\infty_0([0,\e],[0,+\infty))$,
such that
$$ \int_\R \beta_\e(t)\,dt=1,$$
and a solution~$u^\e$ of \eqref{eq-sing-pert}
in~$(-\delta,\delta)$, with
$$ \lim_{\e\to0^+} \|(u^\e)''\|_{L^\infty((-\delta,\delta))}=+\infty.$$
\end{theorem}

We also point out the following example
of smooth solutions of equations like~\eqref{eq-sing-pert},
which are uniformly small but do not possess uniform first derivative bounds.
In this example,
the forcing term~$\beta_\e$ satisfies the scaling properties in~\eqref{betaeps},
but~$\beta$ does not satisfy the structural assumptions.

\begin{theorem}\label{EXAMPLE}
There exists~$\beta\in C^\infty(\R,[0,+\infty))$ such that~$\beta=0$ in~$(-\infty,0]$
for which the following statement holds true.

For every~$\e_0\in(0,1]$ there exist~$\e\in(0,\e_0]$ and~$u^\e\in C^\infty(\R)$ 
such that
\begin{eqnarray}
&& \label{90:A:001} 2(u^\e)''''=-\beta_\e(u) {\mbox{ in }}\R,\\
&& \label{90:A:001bis} u^\e=0 {\mbox{ in }}(-\infty,0],\\
&& \label{90:A:002} \sup_{{x\in\R}}|u^\e(x)|\le\e,\\
&& \label{90:A:003} \sup_{{\e\in(0,1]}}|(u^\e)'(1)|=+\infty.
\end{eqnarray}
Here above, $\beta_\e$ is as in \eqref{betaeps}.
\end{theorem}

\medskip

The paper is organized as follows. In Section~\ref{sec:mot} we provide a
simple motivation for~\eqref{eq-sing-pert} based on
a game-theoretic model.  
Section~\ref{sec:sing} contains the proof of the convergence result
in Theorem~\ref{thm-eps-compactness}.

Section~\ref{integral-id} is devoted to the proof of the integral identity 
stated in Lemma~\ref{IDENDOM}. Suitable choices of the test function
in~\eqref{DOMAIN} provide the cornerstone 
to prove Theorem~\ref{lemma:F} in Section~\ref{PF:MO}.
Section~\ref{SHH:SSSA}
is devoted to the proof of Theorem~\ref{thm:strong-lap}.
Then, Theorem~\ref{ALFAGA} is proved in Section~\ref{sec:det}.

Section~\ref{sec:contr} contains the
counterexamples to the
uniform $C^{1,1}$ bounds stated in Theorems~\ref{CONTR}
and~\ref{EXAMPLE}.

The paper ends with an appendix which
provides some decay estimates for the gradient and the Hessian of
solutions of~\eqref{eq-sing-pert}.

\section{Motivations: a simple game-theoretic model for~\eqref{eq-sing-pert}}
\label{sec:mot}

We point out that there is a simple interpretation of~\eqref{eq-sing-pert} which comes from game theory and which can somehow favor the intuition of the problem.
Let us suppose to run a Gaussian stochastic process
in a Cartesian lattice (say, a random walk)
of small step scale~$h$.
The process starts at some point in a given domain~$\Omega$
and there is a prize~$u_0$ assigned
at the boundary. Let us also suppose that there is a penalization function~$v=v(x,t)$ which makes the player pay something
till it exits the domain~$\Omega$ (of course, the ``prize''
$u_0$ can also attain negative values, and the penalization~$v$
can also attain positive values, hence the game can also penalize exits and compensate for remaining in the domain).
More precisely, if the process exits the domain at
a point~$x\in\partial\Omega$, then the player obtains an award~$u_0(x)$; in addition, if the player exits at time~$T$
by following a trajectory~$x:[0,T)\to\Omega$, it has to pay
a fee quantified by
$$\int_0^T v(x(\theta),\theta)\,d\theta.$$

A natural question in this model is: assuming that the time step~$\tau$
in which the random walk 
takes place is quadratic with respect to the spacial scale, i.e.
\begin{equation}\label{TAUh}
\tau=h^2,\end{equation} and~$u=u(x,t)$ denotes the expected
value to win for a player situated at a point~$x\in (h\Z^n)\cap\Omega$ at time~$t\in\tau\N$,
how to describe~$u$ with a good approximation?

For this, we give a heuristic, but hopefully convincing
argument, not indulging in rigorous convergence details
(see e.g.~\cite{MR2584076} for related discussions).
First of all, one can consider that the 
expected winning
value for a player situated at point~$x$ at time~$t+\tau$
is equal to the
expected winning
values for a player at time~$t$ who is situated at points
reachable by the random walk in one iteration
(that is,~$x\pm he$, with~$e$ being an element
of the Euclidean basis~$\{e_1,\dots,e_n\}$),
weighted by the probability that such jumping occurred (that is~$1/2n$, since the process can go in each coordinate direction),
plus the running cost prescribed by the penalization~$v$,
that is, 
$$\int_t^{t+\tau} v(x(\theta),\theta)\,d\theta=
v(x,t)\,\tau+o(\tau),$$ assuming~$\tau$ small enough.
To write this concept in a formula, assuming also~$u$
sufficiently smooth,
we have that
\begin{eqnarray*}&&
u(x,t+\tau) \\&=&
\frac{1}{2n}\,\sum_{i=1}^n \Big( u(x+he_i,t)
+u(x-he_i,t)\Big)
-v(x,t)\,\tau+o(\tau)
\\ &=&
\frac{1}{2n}\,\sum_{i=1}^n 
\left( u(x,t)+h\nabla u(x,t)\cdot e_i+\frac{h^2}2
D^2u(x,t) e_i\cdot e_i+
u(x,t)-h\nabla u(x,t)\cdot e_i+\frac{h^2}2
D^2u(x,t) e_i\cdot e_i
\right)\\&&\qquad
-v(x,t)\,\tau+o(\tau)+o(h^2)\\
&=&
\frac{1}{2n}\,\sum_{i=1}^n 
\Big( 2u(x,t)+ h^2 D^2u(x,t) e_i\cdot e_i\Big)
-v(x,t)\,\tau+o(\tau)+o(h^2)
\\
&=& u(x,t)+
\frac{1}{2n}\,h^2\sum_{i=1}^n 
D^2u(x,t) e_i\cdot e_i
-v(x,t)\,\tau+o(\tau)+o(h^2)
.\end{eqnarray*}
Hence, recalling~\eqref{TAUh},
\begin{eqnarray*}
\frac{u(x,t+\tau)-u(x,t)}{\tau}&=&
\frac{1}{2n}\,\sum_{i=1}^n D^2u(x,t) e_i\cdot e_i
-v(x,t)+o(1)\\&
=&\frac{1}{2n}\,\sum_{i=1}^n\Delta u (x,t)
-v(x,t)+o(1)
,\end{eqnarray*}
that is, in the limit as~$\tau\to0^+$,
\begin{equation}\label{SLAps:1}
\left\{\begin{matrix}
\partial_t u (x,t)= \Delta u(x,t) - v(x,t) & {\mbox{ if $x\in\Omega$ and~$t>0$,}}\\
u(x,t)=u_0(x) &{\mbox{ if $x\in \partial \Omega$.}}
\end{matrix}\right.\end{equation}
Then, one can also consider the case in which the domain
penalization fee~$v$ is not deterministic but it also depends 
on a stochastic process. For instance, one can prescribe~$v$ to
vanish on the boundary of~$\Omega$ and to evolve
with a random walk in~$\Omega$, which
in addition receives an additional increment of size~$c$
if it travels in a region of the domain on which~$u$
changes its sign (like an ``interface prize'').
This would lead to an equation of the type
\begin{equation}\label{asadAA}
\partial_t v=\Delta v+ c{\mathcal{H}}^{n-1}\Big|_{\partial\{u>0\}},\end{equation}
where the latter can be seen as a $(n-1)$-dimensional measure
sitting on the interface. To avoid such a singular measure,
one can replace it with a mollified version induced by the function~$\beta_\e$ in~\eqref{betaeps},
since this function charges~$O(\e^{-1})$
the regions in which the values of~$u$ range in~$(0,\e)$.
In this way, and taking~$c:=1/2$ for simplicity,
one replaces the singular equation in~\eqref{asadAA} by a regularized version, thus obtaining
\begin{equation}\label{SLAps:2}
\left\{\begin{matrix}
\partial_t v(x,t)=\Delta v(x,t)+\displaystyle\frac{\beta_\e \big(u(x,t)\big)}2 &
{\mbox{ if $x\in\Omega$ and $t>0$,}}\\
v(x,t)=0 &{ \mbox{ if~$x\in\partial \Omega$.}}
\end{matrix}\right. \end{equation}
Of course, the stationary solutions of~\eqref{SLAps:1}
and~\eqref{SLAps:2} are of particular interest and they lead
to the system of equations
\begin{equation}\label{SLAps:3}
\left\{\begin{matrix}
\Delta u(x) = v(x) & {\mbox{ if $x\in\Omega$,}}\\
\Delta v(x)=- \displaystyle\frac{\beta_\e \big(u(x)\big)}2 &
{\mbox{ if $x\in\Omega$,}}\\
\\
u(x)=u_0(x) &{\mbox{ if $x\in \partial \Omega$,}}\\
v(x)=0 &{ \mbox{ if~$x\in\partial \Omega$.}}
\end{matrix}\right.\end{equation}
Substituting~$v$ inside the equations in~\eqref{SLAps:3},
one obtains for~$u=u^\e$
the equation in~\eqref{eq-sing-pert},
which is the main object of investigation of our paper,
with Navier boundary conditions.

\section{Convergence properties: proof of
Theorem~\ref{thm-eps-compactness}}\label{sec:sing}

In this section we will study the minimizers~$u^\e$ of the functional
in~\eqref{DOBB}.
Recalling~\eqref{betaeps} and~\eqref{bepsilon},
we also define
\begin{equation*}
\mathcal B(v):=\int_0^v\beta(t)\,dt,\end{equation*}
and we observe that
$$ {\mathcal{B}}_\e(v)={\mathcal{B}}\left(\frac{v}{\e}\right).$$
In particular, recalling~\eqref{lambda}, we have that,
for any~$x\in\Omega$,
\begin{equation}\label{BESP}
\mathcal B_\e(v(x))=
\begin{cases}
1 & \text{if } v(x)>\e,\\
\\
\displaystyle\int_0^{v(x)/\e}\beta(\tau)\,d\tau & \text{if } v(x)\le \e.
\end{cases}
\end{equation}
Hence 
\begin{equation}\label{bigbetabound}
0\le \mathcal B_\e (v)\le 1,
\end{equation} 
which says that the functions~$\mathcal B_\e$
are uniformly bounded in~$\epsilon$.
{F}rom this,
one can repeat the proof of Theorem~1.1 in~\cite{BIHA}
(see also~\cite{DK, selfdriven})
and obtain that
$$ {\mbox{$\Delta u^\e\in BMO_{\rm{\rm loc}}$, uniformly in $\e$.}}$$
In particular, we find that
\begin{equation}\label{star25}
{\mbox{$u^\e\in W^{2,p}_{\rm{\rm loc}}(\Omega)\cap C_{\rm{\rm loc}}^{1, \alpha}(\Om)$
for every~$\alpha\in(0,1)$ and~$p\in[1,+\infty)$, uniformly in~$\e$.}}\end{equation} 
Moreover, from~\eqref{eq-sing-pert}, it follows that ~$u^\e$
is locally~$C^\infty$ in~$\Om$, with bounds which in general depend on~$\e$.

We stress that estimates that are uniform in~$\e$, as the ones in~\eqref{star25},
are special, they depend on the structure of the problem taken into account,
and they cannot follow from standard elliptic regularity theory (see~\cite{GT}),
as pointed
out in Theorem~\ref{EXAMPLE}.
\medskip 

Now, we want to study the behaviour
of the minimizer~$u^\e$ as~$\e \to 0$. 
We start with the following preliminary convergence result:

\begin{lemma}\label{claim-eps}
For every $v\in \mathcal A$ we have that 
\begin{equation*}
\lim_{\e\to 0} J_\e[v]=J[v].
\end{equation*}
\end{lemma}

\begin{proof}
Recalling the definition of~$\mathcal B_\e$ in~\eqref{bepsilon}, we see that
\begin{equation}\begin{split}\label{deghregbvjds}
J[v]-J_\e[v]=\;&
\int_{\Om}\left({|\Delta v|^2}+ \I{v}\right)-\int_\Om\left(|\Delta v|^2
+\mathcal B_\e(v)\right)\\
=\;&
\int_\Om\Big(\I{v}-\mathcal B_\e(v)\Big)\\
=\;&
\int_{\{0<v<\e\}\cap \Om}\Big(\I{v}-\mathcal B_\e(v)\Big).
\end{split}\end{equation} 
Observe that  
\[
0\le \I{v}-\mathcal B_\e(v)\le 1.
\]
Using this observation together with~\eqref{deghregbvjds}, we conclude that
\[
0\le J[v]-J_\e[v]\le \big|\{0<v<\e\}\cap \Om\big|.
\]
Hence, to complete the proof, it remains to show
that
\begin{equation}\label{9ht97ufhbbc}
\big|\{0<v<\e\}\cap \Om\big|\to 0\quad {\mbox{ as }}\e\to0. 
\end{equation}
For this, let $v\in \mathcal A$. Then $v$ is 
quasicontinuous, i.e. for every~$\sigma>0$ small
there exists a compact set~$E_0$  such that~$v$
is continuous on~$\Om\setminus E_0$ and~$\text{cap}_2(E_0)<\sigma$
(see e.g.~\cite{MR1954868}).

Let $E:=\{x\in \Om\setminus E_0 : v(x)>0\}$.
Then~$E$ is bounded and open. Moreover, we have that
\begin{equation}\label{fjdgbvbj}
\big|\{0<v<\e\}\cap \Om\big| 
= \int_E\big(\chi_{\{v>0\}}-\chi_{\{v\ge \e\}}\big)+\int_{E_0}\chi_{\{0<v<\e\}}.
\end{equation}
Note also that
$$\int_E\chi_{\{v>0\}}\ge \int_E \chi_{\{v\ge \e\}}. $$ 
Thus, taking a sequence $\e_k{{\to}} 0$,
we get from Fatou Lemma
\[
\int_E\chi_{\{v>0\}}=\liminf_{\e_k\to 0}\int_E\chi_{\{v>0\}}\ge \liminf_{\e_k\to 0}\int_E\chi_{\{v\ge\e_k\}}
\ge \int_E \liminf_{\e_k\to 0}\chi_{\{v\ge\e_k\}}\ge \int_E \chi_{\{v>0\}}.
\]
Since $\int_E\chi_{\{v\ge\e\}}$ is non decreasing in $\e$ it follows that 
\[
\lim_{\e\to 0}  \int_E\big(\chi_{\{v>0\}}-\chi_{\{v\ge \e\}}\big)=0. 
\]
{F}rom this and~\eqref{fjdgbvbj} it follows that, for any~$\sigma>0$,
$$ \big|\{0<v<\e\}\cap \Om\big|\le C\sigma,$$
for some~$C>0$. Now the claim in~\eqref{9ht97ufhbbc}
follows if we let $\sigma\to 0$.
This completes the proof of Lemma~\ref{claim-eps}.
\end{proof}

With this, we can now prove the following ``convergence to minimizers'' result:

\begin{lemma}\label{teo:conv}
Suppose that, for any~$k\in\N$, 
\begin{equation}\label{ewiry854guer:PRE}
J_{\e_k}[u^{\e_k}]=\inf_{v\in \mathcal A}J_{\e_k}[v],
\end{equation}
and that~$u^{\e_k}\to u$ locally uniformly on the compact subsets
of~$\overline\Om$ as~$k\to+\infty$. 
Then
$$J[u]=\inf_{v\in \mathcal A}J(v).$$
\end{lemma}

\begin{proof}
Suppose by contradiction that the claim fails. 
Then, there exists~$\widetilde u\in{\mathcal{A}}$ such that 
\begin{equation}\label{ewiry854guer:0}
J[u]-J[\widetilde u]=\delta>0.
\end{equation}
Also, by Lemma~\ref{claim-eps}, we have that $J_{\e_k}[u]\to J[u]$
and $J_{\e_k}[\widetilde u]\to J[\widetilde u]$, as~$\e_k\to0$. Hence, 
for sufficiently small $\e_k$, we have that 
\begin{equation}\label{ewiry854guer:1}
\big|J_{\e_k}[u]-J[u]\big|<\frac{\delta}4 \quad {\mbox{ and }}\quad 
\big|J_{\e_k}[\widetilde u]-J[\widetilde u]\big|<\frac{\delta}4,
\end{equation}
and also 
\[
J_{\e_k}[u]-J_{\e_k}[\widetilde u]>\frac{\delta}2.
\]
{F}rom the proof of Lemma~2.1 in~\cite{BIHA} (see in particular the formula
in display before~(2.5) in~\cite{BIHA}), one can see that~$
\|u^{\e_k}\|_{W^{2,2}(\Om)}\le C$ uniformly in~$\e$, for some~$C>0$.
Moreover, by~\eqref{bigbetabound}, the functions~${\mathcal{B}}_\e$ are
uniformly bounded in $L^\infty(\R)$. 
Therefore we can extract a subsequence, still denoted $u^{\e_k}$, so that 
\begin{itemize}
\item $u^{\e_k}\to u$ locally uniformly in $\Om$,
\item $u^{\e_k}\to u$ weakly in $W^{2, 2}(\Om)$,
\item $\mathcal B_{\e_k}(u^{\e_k})\to \ell$ weak-star in $L^\infty(\Om)$.
\end{itemize}
Note that if $u(x)>0$ at some~$x\in\Omega$,
then $u^{\e_k}(x)>0$ for sufficiently large $k$, possibly depending on~$x$. Hence 
$\ell(x)=1$ if $u(x)>0$, and so $\ell(x)\ge \I u(x)$.
Hence, from Fatou Lemma we have that 
\begin{equation}\label{ewiry854guer}
\liminf_{\e_k\to 0} \int_\Om \mathcal B_{\e_k}(u^{\e_k})\ge \int_\Om \ell(x)\ge \int_{\Om}
\I u.
\end{equation}
Moreover, by~\eqref{ewiry854guer:0} and~\eqref{ewiry854guer:1}, we have that
$$ J[u]-\delta =J[\widetilde u]\ge J_{\epsilon_k}[\widetilde u]-\frac\delta4
\ge J_{\epsilon_k}[u^{\epsilon_k}]-\frac\delta4,$$
where we also used the minimizing property in~\eqref{ewiry854guer:PRE}.

As a consequence,
using the lower semicontinuity of the $L^2$ norm of $\Delta u^{\e_k}$ and
recalling~\eqref{ewiry854guer},
\[ 
J[u]-\delta \ge \liminf_{k\to+\infty} J_{\e_k}[u^{\e_k}]-\frac\delta4\ge J[u] 
-\frac\delta4,
\]
which is a contradiction, and so the proof of Lemma~\ref{teo:conv}
is completed.
\end{proof}

The statement in Theorem~\ref{thm-eps-compactness} is now the
summary of the results obtained in this section, since it follows plainly from~\eqref{star25} and
Lemma~\ref{teo:conv}.

\section{An integral identity for solutions:
proof of Lemma~\ref{IDENDOM}}\label{integral-id}

We provide here the integral relation satisfied
by the
solutions of~\eqref{eq-sing-pert} stated in Lemma~\ref{IDENDOM}.

\begin{proof}[Proof of Lemma~\ref{IDENDOM}]
We write~$u:=u^\e$ for short and we use \eqref{bepsilon} 
to get that 
$$
\nabla\big(\mathcal B_\e(u(x))\big)=
\nabla
\int_0^{u(x)}\beta_\e(t)\,dt=
\beta_\e(u(x))\nabla u(x).
$$
Hence, by the Divergence Theorem,
\begin{equation}\label{OMAMGHABdie}
\begin{split}
\int_\Om \left(|\Delta u|^2+\mathcal B_\e(u)\right)\div \phi\;&=
\int_\Om \div \left(\Big(|\Delta u|^2+\mathcal B_\e(u)\Big) \phi\right)
-\int_\Om (2\Delta u\nabla\Delta u+\beta_\e (u)\nabla u)\cdot\nabla\phi\\
&=-\int_\Om (2\Delta u\nabla\Delta u+\beta_\e (u)\nabla u)\cdot\phi.
\end{split}\end{equation}
On the other hand, in light of~\eqref{eq-sing-pert},
\begin{eqnarray*}
&& -\int_\Om \beta_\e (u)\nabla u\cdot \phi=2
\int_\Om \Delta^2 u\,\nabla u\cdot \phi\\
&&\qquad=2\int_\Om \Delta u\,\Delta(\nabla u\cdot\phi)
=2\sum_{i=1}^n\int_\Om \Delta u\,\Delta\,(\partial_i u\,\phi^i)\\&&\qquad
=2\sum_{i=1}^n\int_\Om \Delta u\,(\partial_i \Delta u\,\phi^i+
\partial_i u\,\Delta\phi^i+2\partial_i \nabla u\cdot\nabla\phi^i)
\\&&\qquad
=2 \int_\Om \Delta u\,\big(\nabla\Delta u\cdot\phi+
\nabla u\cdot\Delta\phi+2 {\rm tr}(D^2 u D\phi)\big)
,
\end{eqnarray*}
which, combined with~\eqref{OMAMGHABdie}, leads to~\eqref{DOMAIN}
after a simplification.
\end{proof}

We observe that
another proof of~\eqref{DOMAIN}
can be performed by a domain perturbation,
looking at
$$u_\eta(x):=u(x+\eta\phi(x)),$$
which can be set into a ``vertical perturbation''
setting
$$\psi_\eta(x):=\frac{u(x+\eta\phi(x))-u_\eta(x)}{\eta},$$
finding that~$u_\eta=u+\eta\psi_\eta$ and thus computing the first
order perturbation in~$\eta$ of the energy functional in~\eqref{DOBB}
gives another proof of~\eqref{DOMAIN}.

We also point out that, as~$\e\to0^+$, formula~\eqref{DOMAIN}
also recovers formula~(4.4) in~\cite{BIHA}.

\section{Monotonicity formula: proof of Theorem~\ref{lemma:F}}\label{PF:MO}

This section is devoted to the proof of Theorem~\ref{lemma:F}.
As already mentioned, 
the strategy here is obtain suitable integral cancellations
by a series of careful integration by parts.
We start with some general computations valid in~$\R^n$.
In this part of the paper, for the sake of shortness, we suppose that
the assumptions of Theorem~\ref{lemma:F} are always satisfied
without further mentioning them. We write~$u:=u^\e$
for the sake of shortness and,
without loss of generality, we also suppose that~$B_2\Subset\Omega$.
Then, we have the following identity:

\begin{lemma}
For every~$r_1$, $r_2\in(0,3/2)$,
\begin{equation}\label{71qy81qush}\begin{split}
4\int_{r_1}^{r_2}R(r)\,dr
+2T(r_2)-2T(r_1)
+D(r_2)-D(r_1)=0,\end{split}\end{equation}
where
\begin{equation}\label{R1r}
\begin{split}
R(r)\,&:=\frac1{r^{n+1}}\sum_{m=1}^n
\int_{B_r} \Delta u\, \nabla u_{m}\cdot e_m
-\sum_{m=1}^n\int_{\partial B_{r}} \Delta u\nabla u_m\cdot
\frac{x^m\,x}{r^{n+2}}\\&=
\frac1{r^{n+1}}
\int_{B_r} |\Delta u|^2
-\frac{1}{r^n}\int_{\partial B_{r}} \Delta u\,\partial^2_r u
,\\
T(r)\,&:=
\sum_{m=1}^n\int_{\partial B_{r}} \Delta u\,u_m\,\frac{x^m}{r^{n+1}}\\
&=\frac1{r^n}\int_{\partial B_{r}} \Delta u\,\partial_ru
\\{\mbox{and }} \qquad D(r)\,&:=
\frac1{r^n}\int_{B_r} \big( |\Delta u|^2+\mathcal B_\e(u)\big),
\end{split}\end{equation}
and the notations~$x:=(x^1,\dots,x^n)$
and~$\partial_r:=\frac{x}{|x|}\cdot\nabla$ have been used.
\end{lemma}

\begin{proof} Fix~$r\in(0,3/2)$.
We let~$\delta>0$ (to be taken as small as we wish in what follows),
and consider a smooth function~$\eta=\eta_\delta$ supported in~$B_{r+\delta}$.
We also define~$\phi=(\phi^1,\dots,\phi^n)
:\R^n\to\R^n$ as
\begin{eqnarray*} \R^n\ni x=(x^1,\dots,x^n)\longmapsto\phi^m(x):=x^m\eta(x).\end{eqnarray*}
We observe that~$\phi^m$ is supported in~$B_1$,
as long as $\delta$ is sufficiently small. Consequently,
for any~$m\in\{1,\dots,n\}$,
\begin{equation*}
\int_\Omega \Delta u \,u_m\,\Delta\phi^m=
\int_{\R^n} \Delta u \,u_m\,\Delta\phi^m=-
\int_{\R^n} \nabla(\Delta u \,u_m)\cdot\nabla\phi^m=-\int_{\Omega} \nabla(\Delta u \,u_m)\cdot\nabla\phi^m
,
\end{equation*}
or, in compact notation,
$$ \int_\Omega \Delta u \,\nabla u\cdot\Delta\phi
=-\sum_{m=1}^n \int_{\Omega} \nabla (\Delta u \,u_m)\cdot\nabla \phi^m.$$
{F}rom this and~\eqref{DOMAIN},
we find that
\begin{equation}\label{RAT1w}\begin{split}
0\;=\;& 2\int_\Omega  \Big(2\,{\rm tr} \big(D^2u \,D\phi\big)
+\nabla u \cdot\Delta\phi
\Big)\,\Delta u
-\int_\Omega {\rm div}\phi \,\Big( |\Delta u|^2+ \mathcal B_\e(u)\Big)
\\ =\;&
\int_\Omega  \left(2\sum_{m=1}^n\Big(2\,\Delta u\nabla u_m
-\nabla (\Delta u \,u_m)\Big)\cdot\nabla \phi^m
- {\rm div}\phi \,\Big( |\Delta u|^2+ \mathcal B_\e(u)\Big)\right).
\end{split}\end{equation}
Now, we take~$\eta\in C^\infty_0(B_{r+\delta})$ such as
$$ \eta(x):=\begin{cases} 1 & {\mbox{ if }}x\in B_r,\\
\displaystyle\frac{\delta+r-|x|}{\delta}& {\mbox{ if }}x\in B_{r+\delta-\delta^2}\setminus B_{r+\delta^2},
\end{cases}$$
and~$|\nabla\eta|\le 2/\delta$. 
In this way, we have that
\begin{eqnarray*}&& \nabla \eta(x)=
-\frac{x}{\delta\,|x|}\qquad{\mbox{ for all }}x\in B_{r+\delta-\delta^2}\setminus B_{r+\delta^2}\\
{\mbox{and }}&&\nabla\phi^m(x)=e_m\eta(x)-\frac{x^m\,x}{\delta\,|x|}
\qquad{\mbox{ for all }}x\in B_{r+\delta-\delta^2}\setminus B_{r+\delta^2},
\end{eqnarray*}
which also gives that
$$ {\rm div}\phi(x)=
n\eta(x)-\frac{|x|}{\delta}
\qquad{\mbox{ for all }}x\in B_{r+\delta-\delta^2}\setminus B_{r+\delta^2}.$$
Moreover, we see that~$\nabla\phi^m=e_m$ in~$B_r$
and~$|\nabla\phi^m|\le 4/\delta$ in~$(
B_{r+\delta}\setminus
B_{r+\delta-\delta^2})\cup( B_{r+\delta^2}\setminus B_r)$.
As a consequence, we obtain that
\begin{eqnarray*}&&
\int_\Om \Big(2\,\Delta u\nabla u_m
-\nabla (\Delta u \,u_m)\Big)\cdot\nabla \phi^m\\
&=&\int_{B_r}\Big(2\,\Delta u\nabla u_m
-\nabla (\Delta u \,u_m)\Big)\cdot e_m\\&&\quad+
\int_{ B_{r+\delta-\delta^2}\setminus B_{r+\delta^2} } 
\Big(2\,\Delta u\nabla u_m
-\nabla (\Delta u \,u_m)\Big)\cdot\left(e_m\eta(x)-\frac{x^m\,x}{\delta\,|x|}\right)
+O(\delta)\\
&=&\int_{B_r} \Big(2\,\Delta u\nabla u_m
-\nabla (\Delta u \,u_m)\Big)\cdot e_m-
\int_{ \partial B_{r} } 
\Big(2\,\Delta u\nabla u_m
-\nabla (\Delta u \,u_m)\Big)\cdot\frac{x^m\,x}{r}
+O(\delta)
\end{eqnarray*}
and 
\begin{eqnarray*}
&&\int_\Om
{\rm div}\phi \,\Big( |\Delta u|^2+ \mathcal B_\e(u)\Big)\\
&=&n\int_{B_r}
\Big( |\Delta u|^2+ \mathcal B_\e(u)\Big)+
\int_{ B_{r+\delta-\delta^2}\setminus B_{r+\delta^2} } 
\Big( |\Delta u|^2+ \mathcal B_\e(u)\Big)\left(
n\eta(x)-\frac{|x|}{\delta}\right)+O(\delta)\\
&=& n\int_{B_r}
\Big( |\Delta u|^2+ \mathcal B_\e(u)\Big)-r
\int_{ \partial B_r } 
\Big( |\Delta u|^2+ \mathcal B_\e(u)\Big)+O(\delta).
\end{eqnarray*}
We insert these two pieces of information into~\eqref{RAT1w},
and we send~$\delta\to0^+$. In this way, we see that
\begin{equation}\label{rnao}
\begin{split}
0\;&=\;
2\sum_{m=1}^n\left[\int_{B_r} \Big(2\,\Delta u\nabla u_m
-\nabla (\Delta u \,u_m)\Big)\cdot e_m-
\int_{ \partial B_{r} } 
\Big(2\,\Delta u\nabla u_m
-\nabla (\Delta u \,u_m)\Big)\cdot\frac{x^m\,x}{r}\right]\\&\qquad-
n\int_{B_r}
\Big( |\Delta u|^2+ \mathcal B_\e(u)\Big)+r
\int_{ \partial B_r } 
\Big( |\Delta u|^2+ \mathcal B_\e(u)\Big)\end{split}\end{equation}
Now, recalling~\eqref{R1r}, we see that
$$ D'(r)=
\frac1{r^n}\int_{\partial B_r} \big( |\Delta u|^2+\mathcal B_\e(u)\big)-
\frac{n}{r^{n+1}}\int_{B_r} \big( |\Delta u|^2+\mathcal B_\e(u)\big),
$$
and hence we can write~\eqref{rnao}
as
\begin{equation}\label{rnao2}
\begin{split}
0=\frac{2}{r^{n+1}}\sum_{m=1}^n\left[\int_{B_r} \Big(2\,\Delta u\nabla u_m
-\nabla (\Delta u \,u_m)\Big)\cdot e_m-
\int_{ \partial B_{r} } 
\Big(2\,\Delta u\nabla u_m
-\nabla (\Delta u \,u_m)\Big)\cdot\frac{x^m\,x}{r}\right]+
D'(r).\end{split}\end{equation}
We also point out that
\begin{eqnarray*}&&
\int_{B_r} \nabla (\Delta u \,u_m)\cdot e_m=
\int_{B_r} {\rm div} (\Delta u \,u_m\, e_m)=
\int_{\partial B_r} \Delta u \,u_m\, \frac{x^m}{r},
\end{eqnarray*}
and, changing variable,
\begin{eqnarray*}&&
\int_{ \partial B_{r} } \nabla \big(\Delta u(x) \,u_m(x)\big)\cdot\frac{x^m\,x}{r^{n+2}}\,
d{\mathcal{H}}^{n-1}(x)=
\int_{ \partial B_1 } \nabla \big(
\Delta u (ry)\,u_m(ry)\big)\cdot\frac{y^m\,y}{r}\,
d{\mathcal{H}}^{n-1}(y)\\&&\qquad=
\int_{ \partial B_1 } \frac{d}{dr} \big(
\Delta u (ry)\,u_m(ry)\big)\frac{y^m}{r}\,
d{\mathcal{H}}^{n-1}(y)\\&&\qquad=
\frac{d}{dr} 
\int_{ \partial B_1 } \big(
\Delta u (ry)\,u_m(ry)\big)\frac{y^m}{r}\,
d{\mathcal{H}}^{n-1}(y)
+
\int_{ \partial B_1 } \big(
\Delta u (ry)\,u_m(ry)\big)\frac{y^m}{r^2}\,
d{\mathcal{H}}^{n-1}(y)
\\&&\qquad=
\frac{d}{dr} 
\int_{ \partial B_r } \big(
\Delta u (x)\,u_m(x)\big)\frac{x^m}{r^{n+1}}\,
d{\mathcal{H}}^{n-1}(x)
+
\int_{ \partial B_r } \big(
\Delta u (x)\,u_m(x)\big)\frac{x^m}{r^{n+2}}\,
d{\mathcal{H}}^{n-1}(x).
\end{eqnarray*}
These observations and~\eqref{R1r} give that
$$ \sum_{m=1}^n\left[-\frac1{r^{n+1}}\int_{B_r} \nabla (\Delta u \,u_m)\cdot e_m+
\int_{ \partial B_{r} } \nabla (\Delta u \,u_m)\cdot\frac{x^m\,x}{r^{n+2}}\right]=\sum_{m=1}^n
\frac{d}{dr} 
\int_{ \partial B_r } \big(
\Delta u \,u_m\big)\frac{x^m}{r^{n+1}}=T'(r).$$
Thus, inserting this information into~\eqref{rnao2}, we find that
\begin{equation}\label{rnao3}
\begin{split}
0=\frac{4}{r^{n+1}}\sum_{m=1}^n\left[\int_{B_r}\Delta u\nabla u_m\cdot e_m-
\int_{ \partial B_{r} } 
\Delta u\nabla u_m\cdot\frac{x^m\,x}{r}\right]+2T'(r)+
D'(r).\end{split}\end{equation}
{F}rom this and~\eqref{R1r}, we can write
$$ 0=4R(r)+2T'(r)+
D'(r),$$
which, after an integration, gives~\eqref{71qy81qush}, as desired.
\end{proof}

The proof of Theorem~\ref{lemma:F} will also rely on the following
auxiliary result:

\begin{lemma}
In the notation stated by~\eqref{R1r},
we have that, for any~$r_1$, $r_2\in(0,3/2)$,
\begin{equation}\label{DV-1}
\begin{split}&
2T(r_1)-2T(r_2)
+D(r_1)-D(r_2)\\
=\;&
4\int_{r_1}^{r_2}\left(
\frac1{r^{n}}\int_{\partial B_r}\Delta u\,\left(2\frac{u_r}{r}-\partial^2_r u
-2\frac{u}{r^2}\right)
-\frac1{r^{n+1}}\int_{B_r} \Delta^2 u\,u\right)\,dr-4V(r_2)+4V(r_1),\end{split}\end{equation}
where
\begin{equation}\label{DV-2}
V(r):= \frac1{r^{n+1}}\int_{\partial B_r}\Delta u u.
\end{equation}
\end{lemma}

\begin{proof} We observe that
\begin{equation}\label{9:9:ia1ap}
\begin{split}
& \int_{B_r} |\Delta u|^2 = \int_{B_r}\Big( {\rm div}(\Delta u\nabla u)-
\nabla\Delta u\cdot\nabla u\Big)
=\int_{\partial B_r}\Delta u\,u_r-\int_{B_r}\nabla\Delta u\cdot\nabla u\\
&\qquad=\int_{\partial B_r}\Delta u\,u_r-\int_{B_r}{\rm div}(u\nabla\Delta u)
+\int_{B_r} \Delta^2u\,u\\
&\qquad=\int_{\partial B_r}\Delta u\,u_r-\int_{\partial B_r} u\,\Delta u_r
+\int_{B_r} \Delta^2u\,u.
\end{split}\end{equation}
Furthermore, we see that
\begin{eqnarray*}&&
\frac{d}{dr} \left(\frac1{r^{n+1}}\int_{\partial B_r}\Delta u u\right)\\&=&
\frac{d}{dr} \left(\frac1{r^{2}}\int_{\partial B_1}\Delta u(r\theta) u(r\theta)\right)\\&=&
-\frac2{r^{3}}\int_{\partial B_1}\Delta u(r\theta) u(r\theta)
+
\frac1{r^{2}}\int_{\partial B_1}\Delta u_r(r\theta) u(r\theta)
+
\frac1{r^{2}}\int_{\partial B_1}\Delta u(r\theta) u_r(r\theta)\\&=&
-\frac2{r^{n+2}}\int_{\partial B_r}\Delta u\,u
+
\frac1{r^{n+1}}\int_{\partial B_r}\Delta u_r\, u
+
\frac1{r^{n+1}}\int_{\partial B_r}\Delta u\,u_r.
\end{eqnarray*}
This and~\eqref{9:9:ia1ap} give that
\begin{eqnarray*}
&& \frac1{r^{n+1}}\int_{B_r} |\Delta u|^2
-\frac1{r^n}\int_{\partial B_{r}} \Delta u\,\partial^2_r u\\
&=&
\frac1{r^{n+1}}\int_{\partial B_r}\Delta u\,u_r-\frac1{r^{n+1}}\int_{\partial B_r} u\,\Delta u_r
+\frac1{r^{n+1}}\int_{B_r} \Delta^2u\,u-\frac1{r^n}\int_{\partial B_{r}} \Delta u\,\partial^2_r u\\
&=&\frac1{r^{n}}\int_{\partial B_r}\Delta u\,\left(2\frac{u_r}{r}-\partial^2_r u
-2\frac{u}{r^2}\right)
+\frac1{r^{n+1}}\int_{B_r} \Delta^2u\,u
-\frac{d}{dr} \left(\frac1{r^{n+1}}\int_{\partial B_r}\Delta u u\right).
\end{eqnarray*}
Now we
integrate the identity above and recall~\eqref{DV-2},
to conclude that
\begin{equation}\label{PamduST}
\begin{split}&
\int_{r_1}^{r_2}\left( \frac1{r^{n+1}}\int_{B_r} |\Delta u|^2
-\frac1{r^n}\int_{\partial B_{r}} \Delta u\,\partial^2_r u
\right)\,dr
\\ =\;&
\int_{r_1}^{r_2}\left(
\frac1{r^{n}}\int_{\partial B_r}\Delta u\,\left(2\frac{u_r}{r}-\partial^2_r u
-2\frac{u}{r^2}\right)
+\frac1{r^{n+1}}\int_{B_r} \Delta^2u\,u\right)\,dr-V(r_2)+V(r_1)
.\end{split}\end{equation}
Hence, recalling~\eqref{R1r}, we can write~\eqref{PamduST} as
$$ \int_{r_1}^{r_2}R(r)\,dr=
\int_{r_1}^{r_2}\left(
\frac1{r^{n}}\int_{\partial B_r}\Delta u\,\left(2\frac{u_r}{r}-\partial^2_r u
-2\frac{u}{r^2}\right)
-\frac1{r^{n+1}}\int_{B_r} \Delta^2u\,u\right)\,dr-V(r_2)+V(r_1).$$
{F}rom this and~\eqref{71qy81qush} we obtain the desired claim in~\eqref{DV-1}.
\end{proof}

The previous calculations were valid in any dimension~$n$,
and we now restrict to the case $n=2$.

\begin{proof}[Proof of \eqref{MONOFORMULA}]
Using using polar coordinates~$(r,\theta)$, we compute 
\begin{equation}\label{ijNAYNTE}
\begin{split}
-\frac1{r^n}\int_{\partial B_r}\Delta u\,\Big(2\frac{u_r}{r}
- \partial^2_r u- 2\frac{u}{r^2}\Big)
=
& \int_{\partial B_1}\frac1{r}\Delta u\,\Big(
u_{rr}-2\frac{u_r}{r}+2\frac{u}{r^2}\Big)
\\=\;&\int_{\partial B_1}\frac1{r}
\Big( u_{rr}+\frac{u_r}r+\frac{u_{\theta\theta}}{r^2} \Big)\Big(u_{rr}-2\frac{u_r}{r}+2\frac{u}{r^2}\Big)
\\=\;&
A(r)+B(r),
\end{split}\end{equation}
where 
\begin{equation}\label{7wqtfychv78rtef7465y45ihgbksajgdf}
\begin{split}&
A(r):= \int_{\p B_1}\frac1{r^3}u_{\theta\theta}\Big(
u_{rr}-2\frac{u_r}{r}+2\frac{u}{r^2}\Big)\\ {\mbox{and }}\quad&
B(r):=\int_{\p B_1}\frac1{r}\Big( u_{rr}+\frac{u_r}r\Big)\Big(u_{rr}-2\frac{u_r}{r}+2\frac{u}{r^2}\Big)
.\end{split}
\end{equation}
Now we deal with the terms~$A(r)$ and~$B(r)$
separately. To start with, we perform several
integrations by parts that involve the terms related to~$A(r)$. 
We see that
\begin{equation}\label{COMA-MA1}
\begin{split}
\frac1{r^3}\int_{\p B_1}u_{\theta\theta}u_{rr}=
&-
\frac1{r^3}\int_{\p B_1} u_\theta u_{\theta rr}\\
=&
-\frac d{dr}\int_{\p B_1}\frac{u_\theta u_{r\theta}}{r^3}+\int_{\p B_1}\frac{u_{r\theta}^2}{r^3}-3
\int_{\p B_1}\frac{u_\theta u_{\theta r}}{r^4}.
\end{split}
\end{equation}
Similarly, we have that
\begin{equation}\label{COMA-MA2}
-2\int_{\p B_1}\frac1{r^4} u_{\theta\theta} u_r=
2\int_{\p B_1}\frac{u_\theta u_{\theta r}}{r^4}= 2\int_{\p B_1}\frac{u_\theta u_{\theta r}}{r^4}
\end{equation}
and 
\begin{equation}\label{COMA-MA3}
2\int_{\p B_1}\frac1{r^5} u_{\theta\theta} u=-2\int_{\p B_1}\frac{u_\theta^2}{r^5}.
\end{equation}
Combining~\eqref{7wqtfychv78rtef7465y45ihgbksajgdf}
\eqref{COMA-MA1}, \eqref{COMA-MA2} and~\eqref{COMA-MA3},
we find that
\begin{equation}\label{BIR0}
\begin{split}
A(r)=
&
-\frac d{dr}\left(
\int_{\p B_1}\frac{u_\theta u_{r\theta}}{r^3}\right)+\int_{\p B_1}\frac{u_{r\theta}^2}{r^3}-3
\int_{\p B_1}\frac{u_\theta u_{\theta r}}{r^4}
+2\int_{\p B_1}\frac{u_\theta u_{\theta r}}{r^4}
-2\int_{\p B_1}\frac{u_\theta^2}{r^5}\\
=&
-\frac d{dr}\left(\int_{\p B_1}\frac{u_\theta u_{r\theta}}{r^3}\right)
+\int_{\p B_1}\frac{u_{r\theta}^2}{r^3}-
\int_{\p B_1}\frac{u_\theta u_{\theta r}}{r^4}
-2\int_{\p B_1}\frac{u_\theta^2}{r^5}\\
= 
&
-\frac d{dr} \left(\int_{\p B_1} \frac{u_\theta u_{r\theta}}{r^3}\right)
+\int_{\p B_1}\frac1{r^3}\left( u_{\theta r}-\frac{2 u_\theta}r \right)^2
+
3\int_{\p B_1}\frac{u_\theta u_{\theta r}}{r^4}
-6\int_{\p B_1}\frac{u_\theta^2}{r^5}\\
=&
-\frac d{dr} \left( \int_{\p B_1} \frac{u_\theta u_{r\theta}}{r^3}
+ \frac32\int_{\p B_1}\frac{u_\theta^2}{r^4}\right)
+\int_{\p B_1}\frac1{r^3}\left( u_{\theta r}-\frac{2 u_r}r \right)^2
\\=&
-\frac d{dr} \left( \int_{\p B_r} \frac{u_\theta u_{r\theta}}{r^4}
+ \frac32\int_{\p B_r}\frac{u_\theta^2}{r^5}\right)
+\int_{\p B_r}\frac1{r^4}\left( u_{\theta r}-\frac{2 u_r}r \right)^2
.\end{split}
\end{equation}
Now we take into account the term~$B(r)$. To this end,
from~\eqref{7wqtfychv78rtef7465y45ihgbksajgdf},
we see that
\begin{equation}\label{BIR}
\begin{split}
B(r)=&
\int_{\p B_1}\frac1r\left(
u_{rr}^2-\frac{2u_{rr}u_r}r+\frac{2uu_{rr}}{r^2}+\frac{u_r u_{rr}}r-\frac{2 u_r^2}{r^2}
+\frac{2u u_r}{r^3}
\right)\\
=&
\int_{\p B_1}\frac1r\left(
u_{rr}^2-\frac{u_{rr}u_r}r+\frac{2uu_{rr}}{r^2}-\frac{2 u_r^2}{r^2}
+\frac{2u u_r}{r^3}
\right)\\
=&
\int_{\p B_1}\frac1r\left(u_{rr}-\frac{3u_r}r+4\frac u{r^2}\right)^2
+
\frac1r
\left(
\frac{5u_ru_{rr}}r-\frac{6uu_{rr}}{r^2}-\frac{11u_r^2}{r^2}+\frac{26u u_r}{r^3}
-\frac{16u^2}{r^4}
\right)\\
=&
\int_{\p B_1}\frac1r\left(u_{rr}-\frac{3u_r}r+4\frac u{r^2}\right)^2
+\frac d{dr}\left(\int_{\p B_1}\frac{5u_r^2}{2r^2}
-\int_{\p B_1}\frac{6u u_r}{r^3}+
\int_{\p B_1}\frac{4u^2}{r^4}\right)\\
=&
\int_{\p B_r}\frac1{r^2}\left(u_{rr}-\frac{3u_r}r+4\frac u{r^2}\right)^2
+\frac d{dr}\left(\int_{\p B_r}\frac{5u_r^2}{2r^3}
-\int_{\p B_r}\frac{6u u_r}{r^4}+
\int_{\p B_r}\frac{4u^2}{r^5}\right).
\end{split}
\end{equation}
Using~\eqref{BIR0} and~\eqref{BIR}, we conclude that
\begin{equation} \label{ABOAiw}
A(r)+B(r)=\frac1{r^2}
\int_{\p B_r}\left[\left( \frac{u_{\theta r}}r-\frac{2 u_r}{r^2} \right)^2+
\left(u_{rr}-\frac{3u_r}r+4\frac u{r^2}\right)^2\right]+W'(r),
\end{equation}
where
\begin{equation}\label{87333efgrhrggh6r8yfc8i21qrutyerod-12ytge}
W(r):=
\int_{\p B_r}\left( \frac{5u_r^2}{2r^3}-\frac{6u u_r}{r^4}+
\frac{4u^2}{r^5}
- \frac{u_\theta u_{r\theta}}{r^4}
- \frac{3u_\theta^2}{2r^5}\right).
\end{equation}
On the other hand, in view of~\eqref{DV-1} 
and~\eqref{ijNAYNTE},
\begin{eqnarray*}&&
-4V(r_2)+4V(r_1)
+2T(r_2)-2T(r_1)
+D(r_2)-D(r_1)\\&=&-
4\int_{r_1}^{r_2}
\left( \frac1{r^n}\int_{\partial B_r}\Delta u\,\Big(2\frac{u_r}{r}
- \partial^2_r u- 2\frac{u}{r^2}\Big) 
-\frac1{r^{n+1}}\int_{B_r} \Delta^2 u\,u
\right)\,dr
\\&=& 4\int_{r_1}^{r_2}\big(
A(r)+B(r)\big)\,dr
+\int_{r_1}^{r_2}\frac4{r^{3}}\int_{B_r} \Delta^2 u\,u
.\end{eqnarray*}
Consequently, by~\eqref{ABOAiw},
\begin{equation}\label{MAHijyhJAHNNA}\begin{split}
&-V(r_2)+V(r_1)
+\frac{T(r_2)-T(r_1)}2
+\frac{D(r_2)-D(r_1)}4-W(r_2)+W(r_1)\\
=\;&
\int_{r_1}^{r_2}\left\{
\frac1{r^2}
\int_{\p B_r}\left[\left( \frac{u_{\theta r}}r-\frac{2 u_r}{r^2} \right)^2+
\left(u_{rr}-\frac{3u_r}r+4\frac u{r^2}\right)^2\right]\right\}
+\int_{r_1}^{r_2}\frac1{r^{3}}\int_{B_r} \Delta^2 u\,u
.
\end{split}\end{equation}
Recalling \eqref{eq-sing-pert},
\eqref{EMMEDE}, \eqref{R1r}, \eqref{DV-2} and \eqref{87333efgrhrggh6r8yfc8i21qrutyerod-12ytge},
we see that
\begin{eqnarray*}
&&-V(r)
+\frac{T(r)}2
+\frac{D(r)}4-\int_{\p B_r}\left( \frac{5u_r^2}{2r^3}-\frac{6u u_r}{r^4}+
\frac{4u^2}{r^5}
- \frac{u_\theta u_{r\theta}}{r^4}
- \frac{3u_\theta^2}{2r^5}\right)-
\int_{0}^{r}\frac1{\rho^{3}}\int_{B_\rho} \Delta^2 u\,u\\
&=&-\frac1{r^{3}}\int_{\partial B_r}\Delta u u
+\frac1{2r^2}\int_{\partial B_{r}} \Delta u\,\partial_ru
+\frac1{4r^2}\int_{B_r} \big( |\Delta u|^2+\mathcal{B}_\e(u)\big)
\\&&\qquad-
\int_{\p B_r}\left( \frac{5u_r^2}{2r^3}-\frac{6u u_r}{r^4}+
\frac{4u^2}{r^5}
- \frac{u_\theta u_{r\theta}}{r^4}
- \frac{3u_\theta^2}{2r^5}\right)+\int_{0}^{r}\frac1{\rho^{3}}\int_{B_\rho} \beta_\e( u)\,u
\\
&=& E(r).
\end{eqnarray*}
This and~\eqref{MAHijyhJAHNNA} establish the desired
claim in~\eqref{MONOFORMULA}. \end{proof}

%Having completed the proof of \eqref{MONOFORMULA},
%to finish the proof of Theorem~\ref{lemma:F},
%we only need to show that the function~$E$ defined in~\eqref{EMMEDE}
%is bounded, which follows directly from~\eqref{EMMEDE}, since,
%by~\eqref{45:190} and elliptic estimates, we have that
%\begin{equation*}
%|u (x)|\le C\,|x|^2,\qquad |\nabla u(x)|\le C\,|x|\qquad{\mbox{ and }}\qquad
%|D^2 u(x)|\le C,
%\end{equation*}
%for all~$x\in B_{1/2}$, with~$C>0$ possibly depending on~$\e$ here.

\section{Strong convergence of $\Delta u^{\e_j}$ and proof of Theorem \ref{thm:strong-lap}}\label{SHH:SSSA}

This section is devoted to the proof of Theorem \ref{thm:strong-lap}.
To this end, we start by proving the strong convergence claimed
in~\eqref{STRONG:C}.

\begin{proof}[Proof of~\eqref{STRONG:C}]
Our aim is to show that
\begin{equation}\label{6:9192}
\lim_{j\to+\infty} \int_\Omega (\Delta u^{\e_j})^2\le\int_\Omega (\Delta u)^2.
\end{equation}
To prove this, we take~$\eta\in C_0^\infty(\Om,\,[0,1])$, and we see that 
\begin{equation}\label{16:9192}
-\int_\Omega \eta u^{\e_j}\Delta ^2 u^{\e_j}
=
\int_\Omega\Delta u^{\e_j}\Delta(\eta u^{\e_j})=
\int_\Omega \eta (\Delta u^{\e_j})^2+\Delta u^{\e_j}(2\na \eta \na u^{\e_j}+u^{\e_j}\Delta \eta ).
\end{equation}
Moreover, supposing that~$\eta$ is supported in some~$B\Subset\Omega$,
we have that
\begin{eqnarray*}
\left| \int_\Omega \eta u^{\e_j}\Delta ^2 u^{\e_j}\right| &\le&
\frac{\|\eta\|_{L^\infty(B)}}{2}\,\int_B |u^{\e_j}|\,\beta_{\e_j}(u^{\e_j})\\
&=&
\frac{\|\eta\|_{L^\infty(B)}}{2\e_j}\,\int_{B 
\cap\{0<u^{\e_j}\le\e_j\}}
|u^{\e_j}|\,\beta\left(\frac{u^{\e_j}}{\e_j}\right)\\
&\le& \frac{1}{2}\,\|\eta\|_{L^\infty(B)}\,\displaystyle\sup_{[0,1]}\beta
\;|\{0<u^{\e_j}\le\e_j\}\cap B|
\end{eqnarray*}
which is infinitesimal as~$j\to+\infty$, thanks to~\eqref{Cg09}.

Consequently, recalling~\eqref{16:9192},
\begin{equation}\label{fcwit-0-P}
\lim_{j\to+\infty}\int_\Omega\eta(\Delta u^{\e_j})^2=-\lim_{j\to+\infty} \int_\Omega \Delta u^{\e_j}
(2\na \eta \cdot\na u^{\e_j}+u^{\e_j}\Delta \eta ).
\end{equation}
Furthermore,
\begin{eqnarray*}&&
\lim_{j\to+\infty} \left|\int_\Omega \Delta u^{\e_j}\na \eta\cdot \na u^{\e_j}
-\int_\Omega \Delta u\na \eta \cdot\na u\right|\\
&\le& \lim_{j\to+\infty} \left|\int_\Omega (\Delta u^{\e_j}-\Delta u)\na \eta\cdot \na u\right|+
\int_\Omega |\Delta u^{\e_j}|\,|\na \eta|\,| \na u^{\e_j}-\na u|
\\ &\le&\lim_{j\to+\infty} \|\eta\|_{C^1(B)}\,\|\Delta u^{\e_j}\|_{L^2(B)}\,
\|\nabla(u^{\e_j}-u)\|_{L^2(B)}\\
&=&0
,\end{eqnarray*}
thanks to the weak convergence of $\Delta u^{\e_j}$
and the Sobolev embedding, and, similarly,
\begin{eqnarray*}
\lim_{j\to+\infty} \left|\int_\Omega \Delta u^{\e_j}u^{\e_j}\Delta \eta -
\int_\Omega \Delta u^{\e_j}u^{\e_j}\Delta \eta \right|=0.
\end{eqnarray*}
These observations and~\eqref{fcwit-0-P}
yield that
\begin{equation}\label{fcwit-0=Q}
\lim_{j\to+\infty}\int_\Omega\eta(\Delta u^{\e_j})^2=
-\int_\Omega \Delta u(2\na \eta\cdot \na u+u\Delta \eta ).
\end{equation}
By Stampacchia's Theorem, we also know that~$\na u=0$ a.e. in~$\{u=0\}$,
hence we can write~\eqref{fcwit-0=Q} in the form
\begin{equation}\label{fcwit-0}
\lim_{j\to+\infty}\int_\Omega\eta(\Delta u^{\e_j})^2=
-\int_{\{u>0\}} \Delta u(2\na \eta\cdot \na u+u\Delta \eta ).
\end{equation}
Next, we exploit the Sard Theorem in Sobolev spaces (see~\cite{MR1871360})
to see that~$\{u>s_k\}=\partial{\{ u>s_k\}}$ has smooth boundary, for an infinitesimal
sequence~$s_k$.
Hence, after some integrations by parts,
\begin{eqnarray}\nonumber
\int_{\{ u>s_k\}}\eta(\Delta u)^2
&=&
\int_{\{ u=s_k\}}\eta \p_\nu u\,\Delta u-\int_{\{ u>s_k\}}\na u\cdot\na (\Delta u\eta)\\\nonumber
&=& 
\int_{\{ u=s_k\}}\eta \p_\nu u\Delta u-
\int_{\{ u=s_k\}}(u-s_k) \p_\nu (\eta \Delta u )+\int_{\{u>s_k\}} (u-s_k)\Delta (\Delta u\eta)\\\nonumber
&=&
\int_{\{ u=s_k\}}\eta \p_\nu u\Delta u
+\int_{\{  u>s_k\}}(u-s_k)(2\na \Delta u\cdot\na \eta+\Delta u\Delta\eta)\\\nonumber
&=&
\int_{\{ u=s_k\}}\eta \p_\nu u\Delta u
+2\int_{\{ u=s_k\}}( u-s_k)\p_\nu\eta  \Delta u 
-\int_{\{  u>s_k\}} 2\na u \cdot\na \eta \Delta u+(u-s_k)\Delta u\Delta\eta\\\label{fcwit}
&=&
\int_{\{ u=s_k\}}\eta \p_\nu u\Delta u
-\int_{\{  u>s_k\}} 2\na u \cdot\na \eta \Delta u+(u-s_k)\Delta u\Delta\eta,
\end{eqnarray}
where $\nu$ is the exterior normal to~$\{u>s_k\}$.
As a technical detail, we point out that the term~$\partial_\nu \Delta u$
is not really well defined in our setting, hence, to justify~\eqref{fcwit},
one should first approximate~$u$ with a mollification
and then take limit.

Now, we claim that
\begin{equation}\label{7jA923847y6}
\lim_{k\to+\infty}\int_{\{ u=s_k\}}\eta \p_\nu u\Delta u=0.\end{equation}
To see this we recall \eqref{QUAD-lap} and we find that
\begin{equation}\label{jaGSHs}
\begin{split}&
\int_{\{ u=s_k\}}\eta \p_\nu u\Delta u
=
\int_{\{ u=s_k\}}\eta \p_\nu u(\Delta u+C)-C\int_{\{ u=s_k\}}\eta \p_\nu u\\ &\qquad
=
-\int_{\{ u=s_k\}}\eta |\na u|(\Delta u+C)-C\int_{\{ u=s_k\}}\eta \p_\nu u. 
\end{split}\end{equation}
Moreover, we observe 
that
\begin{eqnarray*}
\int_{\{ u=s_k\}}\eta  \p_\nu u
&=&
\int_{\{ u>s_k\}}\div(\eta\na  u)\\
&=&
 \int_{\{ u>s_k\}}\na \eta\cdot\na u+\eta\Delta u,\end{eqnarray*}
and, thus, taking limit,
\begin{equation}\label{fck-wit}
\lim_{k\to+\infty}
\int_{\{ u=s_k\}}\eta  \p_\nu u
=  \int_{\{ u>0\}}\na \eta\cdot\na u+\eta\Delta u
=
\int_{\fb { u}}\eta  \p_\nu u=0.
\end{equation}
Also, in light of~\eqref{QUAD-lap},
\begin{eqnarray*}
0
&\le&
 \int_{\{u=s_k\}}\eta |\na  u|(\Delta u+C)\;\le\;
\Big(\sup_B|\Delta u|+C\Big)\int_{\{u=s_k\}}\eta |\na  u|\\
&=&
-\Big(\sup_B|\Delta u|+C\Big)\int_{\{ u=s_k\}}\eta  \p_\nu u,
\end{eqnarray*}
and therefore
$$ \lim_{k\to+\infty}\int_{\{u=s_k\}}\eta |\na  u|(\Delta u+C)=0,$$
thanks to~\eqref{fck-wit}.

Using this, \eqref{jaGSHs} and~\eqref{fck-wit},
we establish~\eqref{7jA923847y6}, as desired.

Then, combining \eqref{fcwit} with~\eqref{7jA923847y6}, we conclude that
\begin{equation*}\begin{split} 
\int_{\{ u>0\}}\eta(\Delta u)^2\,&=
\lim_{k\to+\infty}
\int_{\{ u>s_k\}}\eta(\Delta u)^2\\&
=-\lim_{k\to+\infty}
\int_{\{  u>s_k\}} 2\na u \cdot\na \eta \Delta u+(u-s_k)\Delta u\Delta\eta\\&
=-
\int_{\{  u>0\}} 2\na u \cdot\na \eta \Delta u+u\Delta u\Delta\eta.
\end{split}\end{equation*}
{F}rom this and~\eqref{fcwit-0}, we see that 
\begin{equation}\label{098708765}
\lim_{j\to +\infty}\int_\Om\eta(\Delta u^\e_j)^2=\int_{\{ u>0\}}\eta(\Delta u)^2\le \int_{\Om}( \Delta u)^2.
\end{equation}
Furthermore, fixing~$\delta>0$ and taking~$\eta$ such that~$|\Omega\setminus
\{\eta=1\}|\le\delta$, recalling~\eqref{QUAD-lap}
we see that
$$
\left|\int_\Om(1-\eta)(\Delta u^\e_j)^2\right|\le
C^2\delta.$$
{F}rom this and~\eqref{098708765} we thereby obtain that
$$
\lim_{j\to +\infty}\int_\Om
(\Delta u^\e_j)^2\le C^2\delta+ \int_{\Om}( \Delta u)^2.
$$ 
Hence, by taking~$\delta$ as small as we wish, we complete the proof of~\eqref{6:9192}.

The weak convergence of~$\Delta u^{\e_j}$ also implies that
\begin{eqnarray*}&&
\lim_{j\to+\infty} \int_\Omega (\Delta u^{\e_j})^2=
\lim_{j\to+\infty} \int_\Omega (\Delta u^{\e_j}-\Delta u)^2+2\Delta u^{\e_j}\Delta u-(
\Delta u)^2
\\ &&\qquad\ge
\lim_{j\to+\infty} \int_\Omega 2\Delta u^{\e_j}\Delta u-(
\Delta u)^2=
\int_\Omega (\Delta u)^2.
\end{eqnarray*}
This and~\eqref{098708765} give that
$$ \lim_{j\to+\infty} \int_\Omega (\Delta u^{\e_j})^2=
\int_\Omega (\Delta u)^2,$$
which in turn implies~\eqref{STRONG:C}.
\end{proof}

\begin{proof}[Strong convergence of Hessian and proof of \eqref{STRONG:H}]
It is easy to check that 
\begin{eqnarray}
\int_\Om u_{ii}^\e u_{jj}^\e\eta
&=&
-\int_\Om u_i^\e(u_{jji}^\e\eta+u_{jj}^\e\eta_i)=
-\int_\Om u_i^\e u_{jj}^\e\eta_i+\int_\Om (u_{ij}^\e)^2\eta+u_i^\e\eta_ju_{ij}^\e.
\end{eqnarray}
Hence the strong convergence of the Hessian follows from 
the strong convergence of the Laplacian in~\eqref{STRONG:C}. 
Taking limits, this proves~\eqref{STRONG:H}. 
\end{proof}

\begin{proof}[Proof of the boundedness of $E$, of~\eqref{STRONG:EE},
and of~\eqref{VSmciiririMONOFORMULA}]
By~\eqref{QUAD}, we know that
\begin{equation*}
|u (x)|\le \wh C\,|x|^2,\qquad |\nabla u(x)|\le \wh C\,|x|\qquad{\mbox{ and }}\qquad
|D^2 u(x)|\le \wh C,
\end{equation*}
for all~$x\in B_{1/2}$, for a suitable~$\wh C>0$.
This gives that the function~$E$ in~\eqref{67A02333} is well defined and bounded.

We now prove~\eqref{STRONG:EE}. This
is somehow a delicate point, since one cannot simply take the limit of
the function~$E^\e$ since the last term in~\eqref{EMMEDE}
is not necessarily infinitesimal in~$\e$ (this possible
pathology can be understood,
for instance, by making a direct computation assuming that~$u^\e$ is quadratic).
To cope with this difficulty, it is convenient to define
\begin{eqnarray*} \tilde E^\e(r)&:=&
\int_{\partial B_r}
\left(\frac{\Delta u^\e\,u^\e_r}{2r^2}-
\frac{5(u^\e_r)^2}{2r^3}
-\frac{\Delta u^\e u^\e}{r^3}+\frac{6u^\e u^\e_r}{r^4}+
\frac{u^\e_{\theta}u^\e_{\theta r}}{r^4}
-\frac{4(u^\e)^2}{r^5}
-\frac{3(u_\theta^\e)^2}{2r^5}
\right)\\&&\qquad\qquad\qquad
+\frac1{4r^2}\int_{B_r} \big( |\Delta u^\e|^2+\mathcal{B}_\e(u^\e)\big).\end{eqnarray*}
By~\eqref{EMMEDE}, we have that
\begin{equation} \label{C6Aelel}
E^\e(r)=\tilde E^\e(r)+
\int_{0}^{r}\frac1{\rho^{3}}\int_{B_\rho} \beta_\e(u^\e)\,u^\e
.\end{equation}
Moreover, by the strong convergence of the Hessian that we have just proved, we know
that
\begin{equation}\label{C6Aelel2}
\lim_{j\to+\infty} \tilde E^{\e_j}(r)=E(r).
\end{equation}
We now fix~$\tau_2>\tau_1>0$, with~$B_{\tau_2}\Subset\Omega$. Then, we have that
\begin{eqnarray*}\left|
\int_{\tau_1}^{\tau_2}\frac1{\rho^{3}}\int_{B_\rho} \beta_\e(u^\e)\,u^\e\right|&\le&
\int_{\tau_1}^{\tau_2}\frac1{\rho^{3}}\int_{B_\rho\cap\{ 0<u^\e\le\e\}}
\beta\left(\frac{u^\e}\e\right)\,\frac{
u^\e}{\e}\\&\le& \displaystyle\sup_{[0,1]}\beta\,\frac{1}{\tau_1^{3}}\,
\int_{\tau_1}^{\tau_2}|B_\rho\cap\{ 0<u^\e\le\e\}|
\\&\le& \displaystyle\sup_{[0,1]}\beta\,\frac{\tau_2-\tau_1}{\tau_1^{3}}\,
|B_{\tau_2}\cap\{ 0<u^\e\le\e\}|.
\end{eqnarray*}
As a consequence, by~\eqref{Cg09}, 
$$ \lim_{j\to+\infty}\int_{\tau_1}^{\tau_2}\frac1{\rho^{3}}\int_{B_\rho} \beta_{\e_j}(u^{\e_j}
)\,u^{\e_j}=0.$$
Using this, \eqref{C6Aelel} and~\eqref{C6Aelel2}, we thus conclude that
\begin{equation}\label{879183:921324}\begin{split}& E(\tau_2)-E(\tau_1)=
\lim_{j\to+\infty} \tilde E^{\e_j}(\tau_2)-\tilde E^{\e_j}(\tau_1)=
\lim_{j\to+\infty}
E^{\e_j}(\tau_2)-E^{\e_j}(\tau_1)-
\int_{\tau_1}^{\tau_2}\frac1{\rho^{3}}\int_{B_\rho} \beta_{\e_j}(u^{\e_j})\,u^{\e_j}\\&\qquad=
\lim_{j\to+\infty}
E^{\e_j}(\tau_2)-E^{\e_j}(\tau_1).\end{split}\end{equation}
{F}rom this and~\eqref{MONOFORMULA}
we obtain~\eqref{STRONG:EE}, as desired.

Also, using~\eqref{879183:921324}, \eqref{MONOFORMULA}
and the strong convergence of the Hessian,
we obtain~\eqref{VSmciiririMONOFORMULA}.
\end{proof}

\begin{proof}[Proof of \eqref{VSmciiririMONOFORMULA:2}]
If~$E$ is constant in~$(0,\tau)$, we deduce
from~\eqref{VSmciiririMONOFORMULA} that
\begin{eqnarray*}
&& -\frac{\partial}{\partial \theta}
\left(-\frac{u_r}{r}+ \frac{2u}{r^2}\right)
=\frac{u_{r\theta}}{r^2}-\frac{2u_\theta}{r}=0\\
{\mbox{and }} && -r\,\frac{\partial}{\partial r}
\left(-\frac{u_r}{r}+\frac{2u}{r^2}\right)=u_{rr}-\frac{3u_r}{r}+
\frac {4u}{r^2}=0.
\end{eqnarray*}
As a consequence, we have that
$$\na \left(-\frac{u_r}{r}
+2 \frac{u}{r^2}\right)=0,$$ which implies that the function~$
-\frac{u_r}{r}+ \frac{2u}{r^2}$ is constant for~$|x|\in(0,\tau)$.

Accordingly, we see that
\begin{equation}\label{ODE} -\frac{u_r}{r}+ \frac{2u}{r^2}=c,\end{equation}
for some~$c\in\R$.
Let now
\begin{equation}\label{ODE2} v(r,\theta):=u(r,\theta)+cr^2\log r.\end{equation}
{F}rom~\eqref{ODE}, we have
$$ v_r=u_r+2cr\log r+cr=\frac{2u}r+2cr\log r=\frac{2v}r.
$$
Integrating this equation, fixed~$\bar r\in(0,\tau)$, we find that
$$ v(r,\theta)=\frac{r^2\,v(\bar r,\theta)}{\bar r^2}.$$
This and~\eqref{ODE2} give that
$$ u(r,\theta)=\frac{r^2\,v(\bar r,\theta)}{\bar r^2}-cr^2\log r.$$
Hence, recalling~\eqref{QUAD},
$$ C\ge\frac{|u(r,\theta)|}{r^2}\ge |c|\,|\log r|-\frac{|v(\bar r,\theta)|}{\bar r^2},$$
and therefore
$$ |c|\le\lim_{r\to0} \frac{|v(\bar r,\theta)|}{\bar r^2\,|\log r|}+\frac{C}{|\log r|}=0.$$
This gives that~$c=0$ and in this way
we can write~\eqref{ODE} as~$ -\frac{u_r}{r}+ \frac{2u}{r^2}=0$,
or, equivalently,~$\na u (x)\cdot x=2u(x)$ for any~$x\in B_\tau$. 
The latter is
the Euler equation for homogeneous functions of degree two, and accordingly
we find that~$u$ is necessarily homogeneous of degree two.\end{proof}

\begin{proof}[Proof of~\eqref{thm:strong-lap-EQ}] 
The proof of~\eqref{thm:strong-lap-EQ} is now standard (for instance,
one can repeat the argument in the proof of Theorem~1.14 in~\cite{BIHA}).
The proof of Theorem~\ref{thm:strong-lap} is thereby complete.
\end{proof}

\section{Quadratic detachment: proof of Theorem~\ref{ALFAGA}}
\label{sec:det}

The proof of Theorem~\ref{ALFAGA} relies on the integral
identity in Lemma~\ref{IDENDOM}, and it goes as follows:

\begin{proof}[Proof of Theorem~\ref{ALFAGA}] We let~$\psi\in C^\infty_0(\Omega)$ and we exploit~\eqref{DOMAIN}
with~$\phi(x):=\big(\psi(x),0,\dots,0\big)$. In this way, we obtain that
\begin{equation}\label{DOMsdfgAIN} \int_\Omega \left[2\left(2\sum_{j=1}^n u^\e_{1j}\psi_j
+u^\e_1\Delta\psi\right)\,\Delta u^\e
-\psi_1\,\Big( |\Delta u^\e|^2+ \mathcal B_\e(u^\e)\Big)\right]=0.\end{equation}
We also remark that
\begin{equation}\label{CONxS:2}
\lim_{\e\to0^+} {\mathcal{B_\e}}\big(u^\e(x)\big)=
\begin{cases}
1 & {\mbox{ if }} x\in\{u>0\},\\
0 & {\mbox{ if }} x\in\{u<0\}.\end{cases}                             
\end{equation}
Indeed, if~$u(x)>0$, we have that~$u^\e(x)>\frac{u(x)}{2}>\e$ if~$\e$ is small enough and hence,
in view of~\eqref{BESP}, we know that~${\mathcal{B_\e}}\big(u^\e(x)\big)=1$.
Conversely, if~$u(y)<0$, we have that~$u^\e(y)<0$ for small~$\e$ and thus
$$ {\mathcal{B_\e}}\big(u^\e(y)\big)=\int_0^{\frac{u^\e(y)}\e}\beta(t)\,dt=0,$$
since~$\beta=0$ in~$\left( \frac{u^\e(y)}\e,0\right)$. These observations
establish~\eqref{CONxS:2}.

Thanks to~\eqref{bigbetabound} and the Dominated Convergence Theorem,
we can take limits inside the integral and find that
$$ \lim_{\e\to0^+}
\int_{\Omega\cap\{u\ne0\}} \psi_1\,\mathcal B_\e(u^\e)=
\int_{\Omega\cap\{u>0\}}\psi_1
.$$
Plugging this identity inside~\eqref{DOMsdfgAIN}, and exploiting the convergence
in~\eqref{CONW22}, we conclude that
\begin{equation}\label{DOMsdfgAIN-BIS}\begin{split}
0 \,&=\lim_{\e\to0^+}
\int_\Omega \left[
2\left(2\sum_{j=1}^n u^\e_{1j}\psi_j+u^\e_1\Delta\psi\right)\,\Delta u^\e
-\psi_1\,\Big( |\Delta u^\e|^2+ \mathcal B_\e(u^\e)\Big)
\right]
\\ &=
\int_\Omega \left[
2\left(2\sum_{j=1}^n u_{1j}\psi_j+u_1\Delta\psi\right)\,\Delta u
-\psi_1\,|\Delta u|^2
\right]-\lim_{\e\to0^+}\int_{\Omega\cap\{u=0\}}\psi_1\,\mathcal B_\e(u^\e)
-\int_{\Omega\cap\{u>0\}}\psi_1
.\end{split}\end{equation}
Now, since in all the cases under consideration~$\{u=0\}$ has zero Lebesgue
measure, we can write~\eqref{DOMsdfgAIN-BIS} as
\begin{equation}\label{DOMsdfgAIN-TRIS}
0 =
\int_\Omega \left[
2\left(2\sum_{j=1}^n u_{1j}\psi_j+u_1\Delta\psi\right)\,\Delta u
-\psi_1\,|\Delta u|^2
\right]
-\int_{\Omega\cap\{u>0\}}\psi_1
.\end{equation}
In addition, from~\eqref{CONW22}, we know that
$$ 
u_1=\alpha x_1\chi_{\{x_1>0\}}+\gamma x_1\chi_{\{x_1<0\}}\qquad{\mbox{ and }}\qquad
\Delta u=u_{11}=\alpha\chi_{\{x_1>0\}}+\gamma\chi_{\{x_1<0\}}\qquad{\mbox{ a.e. in $\Omega$,}}$$
and~$u_{j1}=0$ if~$j\ne1$,
therefore~\eqref{DOMsdfgAIN-TRIS} becomes
\begin{equation}\label{DOMsdfgAIN-QUARIS}\begin{split}
0 \,&=
\int_{\Omega\cap\{x_1>0\}} \left[
\left(4\alpha\psi_1+2\alpha x_1\Delta\psi\right)\,\alpha
-\psi_1\,\alpha^2
\right]\\&\quad+\int_{\Omega\cap\{x_1<0\}} \left[
\left(4 \gamma\psi_1+2\gamma x_1\Delta\psi\right)\,\gamma
-\psi_1\,\gamma^2
\right]
-\int_{\Omega\cap\{u>0\}}\psi_1
.\end{split}\end{equation}
Since
$$ \int_{\Omega\cap\{x_1>0\}} x_1\Delta\psi=
-\int_{\Omega\cap\{x_1>0\}} \nabla x_1\cdot\nabla\psi=
-\int_{\Omega\cap\{x_1>0\}} \psi_1,
$$
and similarly
$$ \int_{\Omega\cap\{x_1<0\}} x_1\Delta\psi=
-\int_{\Omega\cap\{x_1<0\}} \psi_1,
$$
we deduce from~\eqref{DOMsdfgAIN-QUARIS} that
\begin{equation}\label{DOMsdfgAIN-QUIS}
0 =
\alpha^2 \int_{\Omega\cap\{x_1>0\}}\psi_1+\gamma^2
\int_{\Omega\cap\{x_1<0\}} \psi_1
-\int_{\Omega\cap\{u>0\}}\psi_1
.\end{equation}
Now, if $\alpha$, $\gamma>0$, it follows that~$\Omega\cap\{u>0\}=
\Omega\cap\{x_1\ne0\}$ and consequently
\begin{equation}\label{8IUJA:01} \int_{\Omega\cap\{u>0\}}\psi_1=
\int_{\Omega\cap\{x_1\ne0\}}\psi_1=
\int_{\Omega}\psi_1=0.\end{equation}
On the other hand,
if~$\alpha$, $\gamma<0$, it follows that~$\Omega\cap\{u>0\}$ is void,
and consequently
\begin{equation}\label{8IUJA:02} \int_{\Omega\cap\{u>0\}}\psi_1=0.\end{equation}
Hence, in light of~\eqref{8IUJA:01} and~\eqref{8IUJA:02},
we see that
if either~$\alpha$, $\gamma>0$ or~$\alpha$, $\gamma<0$,
we can write~\eqref{DOMsdfgAIN-QUIS} as
$$ 0 =-\alpha^2
\int_{\Omega\cap\{x_1=0\}} \psi+\gamma^2
\int_{\Omega\cap\{x_1=0\}} \psi,$$
which leads to~\eqref{ARRCLA:1} and~\eqref{ARRCLA:1bis} in these cases.

If instead~$\alpha>0$ and~$\gamma\le0$, we have that~$\Omega\cap\{u>0\}=
\Omega\cap\{x_1>0\}$ and consequently, by~\eqref{DOMsdfgAIN-QUIS},
\begin{eqnarray*} 0& =&(\alpha^2-1)
\int_{\Omega\cap\{x_1>0\}} \psi_1+\gamma^2
\int_{\Omega\cap\{x_1<0\}} \psi_1\\
&=&-(\alpha^2-1)\int_{\Omega\cap\{x_1=0\}} \psi
+\gamma^2\int_{\Omega\cap\{x_1=0\}} \psi,
\end{eqnarray*}
which leads to~\eqref{ARRCLA:2}.

Furthermore, if~$\alpha<0$ and~$\gamma=0$, we have that~$
\Omega\cap\{u>0\}$ is void, and hence~\eqref{DOMsdfgAIN-QUIS}
gives that
\[ 0=\alpha^2
\int_{\Omega\cap\{x_1>0\}} \psi_1+\gamma^2
\int_{\Omega\cap\{x_1<0\}} \psi_1
=-\alpha^2\int_{\Omega\cap\{x_1=0\}} \psi.
\]
As a consequence, we find that~$\alpha=0$, against our assumption, and then we obtain~\eqref{ARRCLA:4}, as desired.
\end{proof}

\section{Counterexamples to
uniform $C^{1,1}$ bounds:
proofs of Theorems~\ref{CONTR} and~\ref{EXAMPLE}}\label{sec:contr}

Here we construct the one-dimensional
counterexamples
claimed in Theorem~\ref{CONTR},
using a suitable logarithmic bifurcation
from a quadratic function, and in Theorem~\ref{EXAMPLE}.

\begin{figure}
    \centering
    \includegraphics[width=12cm]{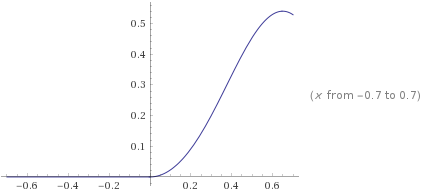}
    \caption{\em {{The counterexample constructed in~\eqref{UJACO3eirjeg} (here, with~$\e:=1/10$).}}}
    \label{CO}
\end{figure}

\begin{proof}[Proof of Theorem~\ref{CONTR}]
We let
\begin{equation}\label{UJACO3eirjeg} u^\e(x):=\begin{cases}
-x^2 \log(\e+x^4) & {\mbox{ if }} x>0,\\
0 & {\mbox{ if }} x\le 0,
\end{cases} \end{equation}
see Figure~\ref{CO}.

We observe that
$$ -\frac{x^4+\e}{2x}\,\frac{du^\e}{dx}(x)=
(x^4+\e) \log(x^4 + \e) + 2 x^4
\le (x^4+\e)\big(\log(x^4 + \e) + 2\big)
<0$$
if~$x\in(0,\sqrt[4]{e^{-2}-\e})$, and so in particular
if~$x\in(0,e^{-3/4})$ as long as~$\e$ is small enough.
This says that the function~$u^\e:(0,e^{-3/4})\to\R$
is strictly increasing and we denote by~$\zeta^\e$ its inverse.
We observe that
$$ u^\e(e^{-3/4})= -e^{-3/2}\log( e^{-3}+\e)\ge 2e^{-3/2}$$
as long as~$\e$ is sufficiently small, hence we can
define~$\zeta_\e$ in~$(0,2e^{-3/2})$.

In this way, for any~$t\in (0,2e^{-3/2})$, we can write that
$$ u^\e(\zeta^\e(t))=t.$$
We let~$\iota_\e\in(0,\e/2)$.
For all~$t\in(0,\iota_\e)$, we define
$$ \beta_\e(t):= -\frac{16 (\zeta^\e(t))^2 \Big(
(\zeta^\e(t))^{12} - 11 \e (\zeta^\e(t))^8 + 135 \e^2 (\zeta^\e(t))^4 - 45
\e^3\Big)
}{\big((\zeta^\e(t))^4 + \e\big)^4}.$$
We notice that
$$ \lim_{t\to0^+}\Big(
(\zeta^\e(t))^{12} - 11 \e (\zeta^\e(t))^8 + 135 \e^2 (\zeta^\e(t))^4 - 45
\e^3\Big)=
- 45\e^3<0,
$$
and hence we can suppose that
$\beta_\e\ge0$ 
in~$(0,\iota_\e)$, provided that~$\iota_\e$ is sufficiently small (possibly in dependence of~$\e$).
We can also extend~$\beta_\e$ to be smooth, zero outside~$(0,\e)$,
and with integral~$1$.

Then, we see that, when~$x>0$ is sufficiently small, 
$$ (u^\e)''''(x)=
-\frac{d^4}{dx^4}\big(x^2 \log(\e+x^4)\big)=
\frac{8 x^2 (x^{12} - 11 \e x^8  + 135\e^2 x^4  - 45 \e^3)}{
(x^4 + \e)^4}=-\frac12\beta_\e(u^\e(x)),$$
and so~$u^\e$ is a local solution of \eqref{eq-sing-pert}.
Nevertheless, it does not possess second derivative bounds
in $L^\infty$
that are uniform in~$\e$ since, for~$x>0$ sufficiently small,
$$ (u^\e)''(x)=
\frac{ 2 \big(6 x^8 + 14 \e x^4  + (x^4 + \e)^2 \log(x^4 + \e)
\big)}{(x^4 + \e)^2}$$
which converges to~$12+  8\log x$
as~$\e\to0^+$, and
$$ (u^\e)''(\sqrt[4]\e)=10  +2\log(2\e),$$
which becomes unbounded as~$\e\to0^+$.
\end{proof}

\begin{proof}[Proof of Theorem~\ref{EXAMPLE}] We take~$\phi\in C^\infty(\R,[0,+\infty))$
such that~$\phi=0$ in~$(-\infty,0]$, $\phi>0$ in~$(0,+\infty)$,
$\phi(k)=k^2$ for every~$k\in\N$ and
$$ \int_\R \phi(t)\,dt=1.$$
Let also
\begin{equation}\label{HAS:3434}
u(x):=\int_{-\infty}^x\phi(t)\,dt=\int_{0}^x\phi(t)\,dt.\end{equation}
We point out that
\begin{equation}\label{UHAN:01}
u=0 {\mbox{ in }}(-\infty,0].
\end{equation}
Moreover, we see that~$u'=\phi>0$ in~$(0,+\infty)$, 
and
\begin{equation}\label{UHAN:02}
\sup_\R u=\lim_{x\to+\infty}u(x)=\int_\R \phi(t)\,dt=1.
\end{equation}
Hence we can invert~$u\big|_{[0,+\infty)}$
and we denote its inverse by~$v$. In this way, $v:[0,1]\to[0,+\infty)$ and
for all~$x\in(0,+\infty)$ we have that
\begin{equation}\label{INVE}
v(u(x))=x.
\end{equation}
Now, for all~$t\in[0,1]$ we define
\begin{equation}\label{betade}
\beta(t):=-2 u''''(v(t)) .\end{equation}
Let also
$$ u^\e(x):= \e\,u\left( \frac{x}{\sqrt\e} \right).$$
Notice that
\begin{equation}\label{UHAN:03}\begin{split}
&{\mbox{if~$x\in(0,+\infty)$, then~$u^\e(x)\in (0,\e)$;}}\\
&{\mbox{moreover~$u^\e=0$ in~$(-\infty,0]$,}}\end{split}
\end{equation}
thanks to~\eqref{UHAN:01} and~\eqref{UHAN:02}.
In particular, the claims in~\eqref{90:A:001bis} and~\eqref{90:A:002}
follow from~\eqref{UHAN:03}.

It also follows from~\eqref{UHAN:03} that
$$ \frac{{u^\e}(x)}{\epsilon}\in(0,1) \qquad{\mbox{for all~$x>0$}},$$
and therefore, by~\eqref{betaeps} and~\eqref{betade}, we have that, for all~$x>0$,
\begin{equation}\label{ujNRFSDV}\begin{split}& -\beta_\e\left( {u^\e}(x)\right)
=-\frac1{\epsilon}\,\beta\left(\frac{{u^\e}(x)}{\epsilon}\right)
=\frac2{\epsilon}\, u''''\left(v\left(\frac{{u^\e}(x)}{\epsilon}\right)\right)
\\&\qquad=\frac2{\epsilon}\, u''''\left(v\left(
u\left( \frac{x}{\sqrt\e} \right)
\right)\right)
=\frac2{\epsilon}\, u''''\left(\frac{x}{\sqrt\e} \right)
=2 (u^\e)''''(x).\end{split}\end{equation}
In addition, from~\eqref{HAS:3434}, we see that
$$ u''''(0)=\phi'''(0)=\lim_{t\to0^-}\phi'''(t)=0.$$
This and \eqref{betade} give that
$$ \beta(0)=-2 u''''(v(0))=-2 u''''(0)=0.$$
Accordingly, from~\eqref{betaeps} and~\eqref{UHAN:03},
for all~$x\le0$,
$$ -2 (u^\e)''''(0)=-\frac2{\epsilon}\, u''''\left(\frac{0}{\sqrt\e} \right)
=0=\beta_\e(0)=\beta_\e\big(u^\e(x)\big).$$
This and~\eqref{ujNRFSDV} establish~\eqref{90:A:001}.

Finally,
$$ (u^\e)'(x)= {\sqrt\e}\,u'\left( \frac{x}{\sqrt\e} \right)=
{\sqrt\e}\,\phi\left( \frac{x}{\sqrt\e} \right).$$
Hence, defining~$\e_k:=k^{-2}$, we see that~$\e_k$ is as small as we wish for large~$k$,
and
$$ (u^{\e_k})'(1)= 
\frac1k\,\phi(k)=k,$$
which gives~\eqref{90:A:003}, as desired.
\end{proof}

\appendix

\section{Decay estimates for the gradient and the Hessian}

Here, we present some decay estimates for the gradient and the Hessian 
of solutions to~\eqref{eq-sing-pert}.

\begin{prop}
Suppose that~$u^\e$ is a solution
of~\eqref{eq-sing-pert} such that 
$|u^\e|\le 1$ in $\Omega$. Let $D\Subset \Om$ and $R_0\in\left(0,\dist(\overline{D}, \p\Omega)\right)$. Suppose that \begin{equation}
\label{lower}
\widehat C:= \inf\limits_{{x\in \overline D}\atop{ \e\in(0,1)}}\fint_{B_{R_0}(x)}\Delta u^\e>-\infty.\end{equation} Then, 
we have 
\begin{equation}\label{0oIANS}
\frac1{R^{n+2}}\int_{B_{R}(x_0)}|\na  u^\e|^2+\frac1{R^n}\int_{B_{R}(x_0)}|D^2 u^\e|^2
\le 
\frac C{R^{n+4}}\int_{B_{4R}(x_0)}(u^\e-m)^2+
\frac{\widehat C}{R^{n+2}}\int_{B_{4R}(x_0)}(u^\e-m), 
\end{equation}
for any~$x_0\in D$ and any~$R\in(0,R_0/4)$,
where 
\begin{equation}\label{Em}m=m^\e:=
\min_{B_{4R}(x_0)} u^\e,
\end{equation}
and
$C>0$ depends only on $n$.
\end{prop}

\begin{proof}
The proof follows from the argument used to prove Lemma A.1 in \cite{BIHA}.
We briefly sketch the argument here. Up to a translation,
we suppose~$x_0:=0$. {F}rom the super biharmonicity of $u^\e$ we get 
\[
0\ge \int_\Om\Delta u^\e\Delta \phi=\sum_{i,j=1}^n\int_\Om u_{ij}^\e\phi_{ij}
\]
for every $\phi\in C_0^\infty(B_{4R},[0,+\infty))$, where two integration by parts are performed in the latter step. 
Choosing~$\phi:=(u^\e-m^\e)\eta^2$, where~$m^\e$ is as in~\eqref{Em},
and~$\eta $ is a standard cut-off
function supported in~$B_{2R}\Subset\Omega$, such that~$ \eta=1$ in~$B_{R}$ and~$
\eta =0$ outside~$B_{2R}$ we get 
\begin{equation}\begin{split}\label{iu67363r:0}
\sum_{i,j=1}^n\int_\Om (u_{ij}^\e)^2\eta^2
\le\; &
\frac C{R^2}\int_{B_{2R}}|\na u^\e|^2
+\frac C{R^4}\int_{B_{2R}}(u^\e-m)^2,
\end{split}\end{equation}
for some universal constant $C>0$ (compare, e.g. with formula~(A.6) in~\cite{BIHA}).

On the other hand, using the mean value property $\Delta u^{\e}(x)\ge \fint_{B_r(x)}\Delta u^\e$ and the lower 
bound in~\eqref{lower}, we obtain the Caccioppoli-type inequality
\begin{equation}\label{NAML2}
\int_{B_{2R}}|\na u^\e|^2
\le \frac{C}{R^2}\,\int_{B_{4R}}
(u^\e - m^\e)^2+C\,\int_{B_{4R}}
(u^\e -m^\e),\end{equation}
see e.g. formula~(7.7) in~\cite{BIHA}. Combining~\eqref{iu67363r:0}
and~\eqref{NAML2}, we
finish the proof. 
\end{proof}

\vfill


\begin{bibdiv}
\begin{biblist}
%%%
%%%\bib{Adams}{book}{
%%%   author={Adams, David R.},
%%%   author={Hedberg, Lars Inge},
%%%   title={Function spaces and potential theory},
%%%   series={Grundlehren der Mathematischen Wissenschaften [Fundamental
%%%   Principles of Mathematical Sciences]},
%%%   volume={314},
%%%   publisher={Springer-Verlag, Berlin},
%%%   date={1996},
%%%   pages={xii+366},
%%%   isbn={3-540-57060-8},
%%%   review={\MR{1411441}},
%%%   doi={10.1007/978-3-662-03282-4},
%%%}
%%%
%	\bib{adams}{article}{
%	   author={Adams, David R.},
%	   author={Vandenhouten, Ronald F.},
%	   title={Stability for polyharmonic obstacle problems with varying
%	   obstacles},
%	   journal={Comm. Partial Differential Equations},
%	   volume={25},
%	   date={2000},
%	   number={7-8},
%	   pages={1171--1200},
%	   issn={0360-5302},
%	   review={\MR{1765144}},
%	}

%\bib{2016arXiv160306819A}{article}{
%   author = {{Aleksanyan}, Goran},
%    title = {Regularity of the free boundary in the biharmonic obstacle problem},
%  journal = {ArXiv e-prints},
%archivePrefix = {arXiv},
%   eprint = {1603.06819},
% primaryClass = {math.AP},
% keywords = {Mathematics - Analysis of PDEs},
%     year = {2016},
%   adsurl = {http://adsabs.harvard.edu/abs/2016arXiv160306819A},
%  adsnote = {Provided by the SAO/NASA Astrophysics Data System}
%}

\bib{MR618549}{article}{
   author={Alt, H. W.},
   author={Caffarelli, L. A.},
   title={Existence and regularity for a minimum problem with free boundary},
   journal={J. Reine Angew. Math.},
   volume={325},
   date={1981},
   pages={105--144},
   issn={0075-4102},
   review={\MR{618549}},
}

%	\bib{Caff-cpde}{article}{
%	   author={Caffarelli, Luis A.},
%	   title={Compactness methods in free boundary problems},
%	   journal={Comm. Partial Differential Equations},
%	   volume={5},
%	   date={1980},
%	   number={4},
%	   pages={427--448},
%	   issn={0360-5302},
%	   review={\MR{567780}},
%	}

\bib{caffa}{article}{
   author={Caffarelli, Luis A.},
   author={Friedman, Avner},
   title={The obstacle problem for the biharmonic operator},
   journal={Ann. Scuola Norm. Sup. Pisa Cl. Sci. (4)},
   volume={6},
   date={1979},
   number={1},
   pages={151--184},
   review={\MR{529478}},
}

\bib{MR620427}{article}{
   author={Caffarelli, Luis A.},
   author={Friedman, Avner},
   author={Torelli, Alessandro},
   title={The free boundary for a fourth order variational inequality},
   journal={Illinois J. Math.},
   volume={25},
   date={1981},
   number={3},
   pages={402--422},
   issn={0019-2082},
   review={\MR{620427}},
}

\bib{MR705233}{article}{
   author={Caffarelli, Luis A.},
   author={Friedman, Avner},
   author={Torelli, Alessandro},
   title={The two-obstacle problem for the biharmonic operator},
   journal={Pacific J. Math.},
   volume={103},
   date={1982},
   number={2},
   pages={325--335},
   issn={0030-8730},
   review={\MR{705233}},
}

%\bib{CLW-uniform}{article}{
%   author={Caffarelli, L. A.},
%   author={Lederman, C.},
%   author={Wolanski, N.},
%   title={Uniform estimates and limits for a two phase parabolic singular
%   perturbation problem},
%   journal={Indiana Univ. Math. J.},
%   volume={46},
%   date={1997},
%   number={2},
%   pages={453--489},
%   issn={0022-2518},
%   review={\MR{1481599}},
%}
%

\bib{MR1871360}{article}{
   author={de Pascale, Luigi},
   title={The Morse-Sard theorem in Sobolev spaces},
   journal={Indiana Univ. Math. J.},
   volume={50},
   date={2001},
   number={3},
   pages={1371--1386},
   issn={0022-2518},
   review={\MR{1871360}},
   doi={10.1512/iumj.2001.50.1878},
}

%\bib{MR1246185}{article}{
%   author={DiBenedetto, E.},
%   author={Manfredi, J.},
%   title={On the higher integrability of the gradient of weak solutions of
%   certain degenerate elliptic systems},
%   journal={Amer. J. Math.},
%   volume={115},
%   date={1993},
%   number={5},
%   pages={1107--1134},
%   issn={0002-9327},
%   review={\MR{1246185}},
%}

\bib{DK}{article}{
   author={Dipierro, Serena},
   author={Karakhanyan, Aram L.},
   title={Stratification of free boundary points for a two-phase variational
   problem},
   journal={Adv. Math.},
   volume={328},
   date={2018},
   pages={40--81},
   issn={0001-8708},
   review={\MR{3771123}},
}

\bib{selfdriven}{article}{
   author={Dipierro, Serena},
   author={Karakhanyan, Aram},
   author={Valdinoci, Enrico},
   title={A nonlinear free boundary problem with a self-driven Bernoulli
   condition},
   journal={J. Funct. Anal.},
   volume={273},
   date={2017},
   number={11},
   pages={3549--3615},
   issn={0022-1236},
   review={\MR{3706611}},
}

\bib{BIHA}{article}{
   author={Dipierro, Serena},
   author={Karakhanyan, Aram},
   author={Valdinoci, Enrico},
   title={A free boundary problem driven by the biharmonic operator},
  journal = {ArXiv e-prints},
archivePrefix = {arXiv},
   eprint = {1808.07696},
 primaryClass = {math.AP},
 keywords = {Mathematics - Analysis of PDEs},
     year = {2018},
   adsurl = {http://adsabs.harvard.edu/abs/2018arXiv180807696D},
  adsnote = {Provided by the SAO/NASA Astrophysics Data System}}

%\bib{EVANS98}{book}{
%   author={Evans, Lawrence C.},
%   title={Partial differential equations},
%   series={Graduate Studies in Mathematics},
%   volume={19},
%   publisher={American Mathematical Society, Providence, RI},
%   date={1998},
%   pages={xviii+662},
%   isbn={0-8218-0772-2},
%   review={\MR{1625845}},
%}

%\bib{Figalli}{article}{
%   author={Figalli, Alessio},
%   title={A simple proof of the Morse-Sard theorem in Sobolev spaces},
%   journal={Proc. Amer. Math. Soc.},
%   volume={136},
%   date={2008},
%   number={10},
%   pages={3675--3681},
%   issn={0002-9939},
%   review={\MR{2415054}},
%}
%
%\bib{frehse}{article}{
%   author={Frehse, Jens},
%   title={On the regularity of the solution of the biharmonic variational
%   inequality},
%   journal={Manuscripta Math.},
%   volume={9},
%   date={1973},
%   pages={91--103},
%   issn={0025-2611},
%   review={\MR{0324208}},
%}
%%%
%%%\bib{Frehse}{article}{
%%%   author={Frehse, J.},
%%%   title={Capacity methods in the theory of partial differential equations},
%%%   journal={Jahresber. Deutsch. Math.-Verein.},
%%%   volume={84},
%%%   date={1982},
%%%   number={1},
%%%   pages={1--44},
%%%   issn={0012-0456},
%%%   review={\MR{644068}},
%%%}

\bib{GANGULI}{book}{
    author = {Ganguli, Ranjan}, 
    title = {Finite element analysis of rotating beams.
{P}hysics based interpolation},
    isbn = {978-981-10-1901-2/hbk; 978-981-10-1902-9/ebook},
    pages = {xii+283},
    date = {2017},
    publisher = {Singapore: Springer},
    review = {ZBl1369.74001}
}

\bib{GAZ}{book}{
   author={Gazzola, Filippo},
   author={Grunau, Hans-Christoph},
   author={Sweers, Guido},
   title={Polyharmonic boundary value problems},
   series={Lecture Notes in Mathematics},
   volume={1991},
   note={Positivity preserving and nonlinear higher order elliptic equations
   in bounded domains},
   publisher={Springer-Verlag, Berlin},
   date={2010},
   pages={xviii+423},
   isbn={978-3-642-12244-6},
   review={\MR{2667016}},
}

%	\bib{Giaq}{book}{
%	   author={Giaquinta, Mariano},
%	   title={Multiple integrals in the calculus of variations and nonlinear
%	   elliptic systems},
%	   series={Annals of Mathematics Studies},
%	   volume={105},
%	   publisher={Princeton University Press, Princeton, NJ},
%	   date={1983},
%	   pages={vii+297},
%	   isbn={0-691-08330-4},
%	   isbn={0-691-08331-2},
%	   review={\MR{717034}},
%	}

\bib{GT}{book}{
   author={Gilbarg, David},
   author={Trudinger, Neil S.},
   title={Elliptic partial differential equations of second order},
   series={Grundlehren der Mathematischen Wissenschaften [Fundamental
   Principles of Mathematical Sciences]},
   volume={224},
   edition={2},
   publisher={Springer-Verlag, Berlin},
   date={1983},
   pages={xiii+513},
   isbn={3-540-13025-X},
   review={\MR{737190}},
   doi={10.1007/978-3-642-61798-0},
}

\bib{MR1954868}{article}{
   author={Kinnunen, Juha},
   author={Latvala, Visa},
   title={Lebesgue points for Sobolev functions on metric spaces},
   journal={Rev. Mat. Iberoamericana},
   volume={18},
   date={2002},
   number={3},
   pages={685--700},
   issn={0213-2230},
   review={\MR{1954868}},
   doi={10.4171/RMI/332},
}

%\bib{MR0350027}{book}{
%   author={Landkof, N. S.},
%   title={Foundations of modern potential theory},
%   note={Translated from the Russian by A. P. Doohovskoy;
%   Die Grundlehren der mathematischen Wissenschaften, Band 180},
%   publisher={Springer-Verlag, New York-Heidelberg},
%   date={1972},
%   pages={x+424},
%   review={\MR{0350027}},
%}

\bib{MR3512704}{article}{
   author={Mardanov, R. F.},
   author={Zaripov, S. K.},
   title={Solution of Stokes flow problem using biharmonic equation
   formulation and multiquadrics method},
   journal={Lobachevskii J. Math.},
   volume={37},
   date={2016},
   number={3},
   pages={268--273},
   issn={1995-0802},
   review={\MR{3512704}},
}

%	\bib{Martio}{article}{
%	   author={Martio, O.},
%	   author={Vuorinen, M.},
%	   title={Whitney cubes, $p$-capacity, and Minkowski content},
%	   journal={Exposition. Math.},
%	   volume={5},
%	   date={1987},
%	   number={1},
%	   pages={17--40},
%	   issn={0723-0869},
%	   review={\MR{880256}},
%	}

\bib{mawi}{article}{
   author={Mawi, Henok},
   title={A free boundary problem for higher order elliptic operators},
   journal={Complex Var. Elliptic Equ.},
   volume={59},
   date={2014},
   number={7},
   pages={937--946},
   issn={1747-6933},
   review={\MR{3195921}},
}

\bib{MR866720}{article}{
   author={McKenna, P. J.},
   author={Walter, W.},
   title={Nonlinear oscillations in a suspension bridge},
   journal={Arch. Rational Mech. Anal.},
   volume={98},
   date={1987},
   number={2},
   pages={167--177},
   issn={0003-9527},
   review={\MR{866720}},
}

\bib{Monneau-W}{article}{
   author={Monneau, R.},
   author={Weiss, G. S.},
   title={An unstable elliptic free boundary problem arising in solid
   combustion},
   journal={Duke Math. J.},
   volume={136},
   date={2007},
   number={2},
   pages={321--341},
   issn={0012-7094},
   review={\MR{2286633}},
}

\bib{novaga1}{article}{
   author={Novaga, Matteo},
   author={Okabe, Shinya},
   title={Regularity of the obstacle problem for the parabolic biharmonic
   equation},
   journal={Math. Ann.},
   volume={363},
   date={2015},
   number={3-4},
   pages={1147--1186},
   issn={0025-5831},
   review={\MR{3412355}},
}

\bib{novaga2}{article}{
   author={Novaga, Matteo},
   author={Okabe, Shinya},
   title={The two-obstacle problem for the parabolic biharmonic equation},
   journal={Nonlinear Anal.},
   volume={136},
   date={2016},
   pages={215--233},
   issn={0362-546X},
   review={\MR{3474411}},
}

\bib{MR1900562}{article}{
   author={Petrosyan, Arshak},
   title={On existence and uniqueness in a free boundary problem from
   combustion},
   journal={Comm. Partial Differential Equations},
   volume={27},
   date={2002},
   number={3-4},
   pages={763--789},
   issn={0360-5302},
   review={\MR{1900562}},
   doi={10.1081/PDE-120002873},
}
		
\bib{pozzo}{article}{
   author={Pozzolini, C\'edric},
   author={L\'eger, Alain},
   title={A stability result concerning the obstacle problem for a plate},
   language={English, with English and French summaries},
   journal={J. Math. Pures Appl. (9)},
   volume={90},
   date={2008},
   number={6},
   pages={505--519},
   issn={0021-7824},
   review={\MR{2472891}},
}

%%	\bib{MR729195}{article}{
%%	   author={Spruck, Joel},
%%	   title={Uniqueness in a diffusion model of population biology},
%%	   journal={Comm. Partial Differential Equations},
%%	   volume={8},
%%	   date={1983},
%%	   number={15},
%%	   pages={1605--1620},
%%	   issn={0360-5302},
%%	   review={\MR{729195}},
%%	}

\bib{SW}{article}{
   author={Sweers, Guido},
   title={A survey on boundary conditions for the biharmonic},
   journal={Complex Var. Elliptic Equ.},
   volume={54},
   date={2009},
   number={2},
   pages={79--93},
   issn={1747-6933},
   review={\MR{2499118}},
}

\bib{MR2584076}{article}{
   author={Valdinoci, Enrico},
   title={From the long jump random walk to the fractional Laplacian},
   journal={Bol. Soc. Esp. Mat. Apl. SeMA},
   number={49},
   date={2009},
   pages={33--44},
   issn={1575-9822},
   review={\MR{2584076}},
}

\bib{MR1620644}{article}{
   author={Weiss, Georg S.},
   title={Partial regularity for weak solutions of an elliptic free boundary
   problem},
   journal={Comm. Partial Differential Equations},
   volume={23},
   date={1998},
   number={3-4},
   pages={439--455},
   issn={0360-5302},
   review={\MR{1620644}},
   doi={10.1080/03605309808821352},
}

\end{biblist}
\end{bibdiv}
\end{document}